\documentclass[12pt]{amsart}
\usepackage{amssymb,amsthm,amsmath,amstext}
\usepackage{mathrsfs}  % allows nice script font for sheaves
\usepackage{bm}        % allows bold italic letters in math mode
\usepackage{mathtools} % allows more extendible arrows
\usepackage{fullpage}
\usepackage{graphicx}
\usepackage[all]{xy}

\mathchardef\mhyphen="2D

\theoremstyle{plain}
\newtheorem{theorem}{Theorem}[section]
\newtheorem{prop}[theorem]{Proposition}
\newtheorem{lemma}[theorem]{Lemma}

\newtheorem{cor}[theorem]{Corollary}

\newtheorem*{nonu-theorem}{Theorem}

\theoremstyle{definition}

\newtheorem{example}[theorem]{Example}

\theoremstyle{remark}
\newtheorem{remark}[theorem]{Remark}
\newtheorem{properties}[theorem]{Properties}

\newcommand\twoheaduparrow{\mathrel{\rotatebox{90}{$\twoheadrightarrow$}}}
\newcommand{\sheaf}[1]{\mathscr{#1}}

\newcommand{\OO}{\mc{O}}

\newcommand{\PP}{\sheaf{P}}
\newcommand{\XX}{\sheaf{X}}

\newcommand{\UU}{\sheaf{U}}
\newcommand{\mc}[1]{\mathcal #1}

\def\F{{\mathbb F}}

\DeclareMathOperator{\Gal}{\mathrm{Gal}}

\DeclareMathOperator{\Hom}{\mathrm{Hom}}

\newcommand{\Z}{\mathbb Z}
\newcommand{\N}{\mathbb N}
\newcommand{\A}{\mathbb A}

\renewcommand{\P}{\mathbb P}
\newcommand{\Q}{\mathbb Q}
\newcommand{\G}{\mathbb G}

\DeclareMathOperator{\Spec}{\mathrm{Spec}}

\newcommand{\et}{\mathrm{\acute{e}t}}

\newcommand{\wh}[1]{\widehat{#1}}

\renewcommand{\labelenumi}{\it\alph{enumi})}
\usepackage{hyperref}
\usepackage{color}
\hypersetup{colorlinks=true,linkcolor=blue,anchorcolor=blue,citecolor=blue,letterpaper=true}

\def\lra{\longrightarrow}

\DeclareFontFamily{U}{wncy}{}
    \DeclareFontShape{U}{wncy}{m}{n}{<->wncyr10}{}
    \DeclareSymbolFont{mcy}{U}{wncy}{m}{n}
    \DeclareMathSymbol{\Sha}{\mathord}{mcy}{"58}

\begin{document}

%========================================================================================
\title[Local-global principles for tori over arithmetic curves]{Local-global principles for tori over arithmetic curves}
%========================================================================================

\author{Jean-Louis Colliot-Th\'el\`ene, David Harbater, Julia Hartmann, \\Daniel Krashen, R. Parimala, and V. Suresh}
\date{July 13, 2020}

\maketitle
\begin{abstract}
In this paper we study local-global principles for tori over semi-global fields, which are one variable function fields over complete discretely valued fields. In particular, we show that for principal homogeneous spaces for tori over the underlying discrete valuation ring, the obstruction to a local-global principle with respect to discrete valuations can be computed using methods coming from patching. We give a sufficient condition for the vanishing of the obstruction, as well as examples where the obstruction is nontrivial or even infinite. A major tool is the notion of a flasque resolution of a torus.
\end{abstract}
%==============================================
\section*{Introduction}
%==============================================

Classical local-global principles assert the existence of rational points on varieties over a global field $F$ under the assumption that the variety has points over certain overfields of~$F$, which are typically obtained via completions.  In recent years, such principles have also been studied over {\em semi-global fields}. A semi-global field is a one-variable function field $F$ over a complete discretely valued field $K$; i.e., a finitely generated extension of $K$ of transcendence degree one in which $K$ is algebraically closed.  For example, see \cite{CHHKPS}, \cite{CTPS1}, \cite{CTPS2}, \cite{HarSz}, \cite{HHK1}, \cite{admis}, \cite{HHK3}, \cite{HHK5}, \cite{Hu}.

In \cite{HHK1}, using patching techniques, a local-global principle was proven for torsors under a  linear algebraic group $G$ over a semi-global field $F$, under the hypothesis that $G$ is connected and rational as an $F$-variety.  The question of whether such a local-global principle also holds without any rationality hypothesis remained open until the paper \cite{CTPS2}, where counterexamples were given for $G=T$ a suitable non-rational torus.  In that example, $F$ does not have a smooth projective model over the valuation ring $R$ of the complete discretely valued field $K$; and the torus $T$ over $F$ does not extend to a torus over any regular projective model of $F$ over $R$.
This raised the following question, given a regular projective $R$-curve $\XX$ with function field $F$:
If one restricts attention
to  $F$-groups $G$ that are the restriction of reductive $\XX$-group schemes, or
of reductive $R$-group schemes, do we have a local-global principle?  If not, can we understand the precise obstruction to the local-global principle?

In the present text, we give answers to these questions in the case of tori. Unlike the situation of global fields, semi-global fields admit several natural collections of overfields due to the richer geometry, and hence there are several different possible local-global principles to consider. The corresponding obstructions are measured by Tate-Shafarevich groups. In our situation, for an $\XX$-torus $T$, we show that the obstruction groups to these different local-global principles all coincide with the Tate-Shafarevich group $\Sha(F,T)$ defined with respect to the completions of~$F$ at discrete valuations (see Theorem~\ref{allshaTequal}).  As a result, in order to prove a local-global principle with respect to these completions of $F$, it suffices to prove such a principle with respect to finitely many overfields arising from patching. Combining this with the notion of a flasque resolution of $T$, we are able to give a double coset formula in terms of the Galois cohomology of a flasque torus $S$ (see Theorem~\ref{H1Scosets}):
\begin{nonu-theorem}
Let $F$ be a semi-global field over a complete discretely valued field with valuation ring~$R$, and let $T$ be a torus over $R$. Let $1 \to S \to Q \to T \to 1$ be a flasque resolution of $T$.
Then there is an isomorphism of abelian groups
$$ \Sha(F,T)\simeq \left.   \prod _U H^1(F_U,S) \middle\backslash \prod_{(U, P)} H^1(F_{U, P},S)
 \middle\slash \prod_P H^1(F_P,S) \right. .$$
 \end{nonu-theorem}
Here the fields $F_P,F_U, F_{U,P}$ form a finite inverse system of fields coming from patching (see Section~\ref{setup}).

This formula implies the finiteness of $\Sha(F,T)$ for certain residue fields $k$ for which $H^1(k,S)$ is always finite (Theorem~\ref{finite}), including when $k$ is finitely generated over ${\mathbb Q}$. Moreover, using properties of flasque tori, it often leads to an exact computation of $\Sha(F,T)$ (see
Section~\ref{computations}, especially Theorems~\ref{Sha via coef sys} and~\ref{shaP1}). Another ingredient for these computations is a general graph theoretic setup that we introduce in Section~\ref{graphs}.

We also provide a sufficient condition for the vanishing of this obstruction in terms of the closed fiber of a normal crossings model $\XX$ of $F$. As in \cite{HHK3}, this condition concerns the reduction graph associated with a normal crossings model, but is more subtle (see Section~\ref{vanishing}; Theorem~\ref{triv Sha monotonic}):
\begin{nonu-theorem}
Let $K$ be a complete discretely valued field with valuation ring~$R$, and let $F$ be a semi-global field over~$K$. Let $\XX$ be a normal crossings model of $F$, and assume that the associated reduction graph $\Gamma$ is a monotonic tree. Then for any $R$-torus $T$,  $\Sha(F,T)$ is trivial.
\end{nonu-theorem}

In Section~\ref{nontriv sha ex}, we give examples where the obstruction is non-trivial, and in one case even infinite. Finally in Section~\ref{triv Sha loop} we give an example in which the obstruction vanishes even though one might have expected it to be nontrivial.

%==================================================
\section{Reminders on Patching}
%==================================================
In this section, we recall the main ingredients from patching and their use in the study of local-global principles over semi-global fields. We begin by recalling the patching setup (\cite{HHK1}, Notation~3.3).

%----------------------------------------------------------------------------------------------------------------------------------------------------------------------------------------------------------------------------------------------------
\subsection{The Patching Setup}\label{setup}
%----------------------------------------------------------------------------------------------------------------------------------------------------------------------------------------------------------------------------------------------------

Let $R$ be a complete discrete valuation ring, $K$ its field of fractions and
$k$ its  residue field. Let $t$ denote a uniformizing parameter for $R$.
Let $F$ be a semi-global field over $K$. A {\em normal model} of~$F$ is an integral $R$-scheme $\XX$ with function field~$F$ that is flat and projective over~$R$ of relative dimension one, and that is normal as a scheme. We write $X$ for the closed fiber $\XX\times_R k$. If $\XX$ is a normal model which is regular as a scheme, we say that $\XX$ is a {\em regular model}. Such a regular model exists by the main theorem in \cite{Lip78}.
In fact, by \cite[page 193]{Lip75}, there exists a regular model $\XX$ for which the reduced closed fiber $X^{\operatorname{red}}$  is a union of regular curves, with normal crossings.
We call such a model a {\em normal crossings model of~$F$}.

Let $\PP$ be a finite nonempty set of closed points of $X$
that contains all the
points of $X^{\operatorname{red}}$ at which different components meet.
Let $\UU$ be the set of irreducible components of $X^{\operatorname{red}} \setminus \PP$. Note that $X^{\operatorname{red}}  \setminus  \PP$ is
an affine curve.
For $U \in \UU$ such a component, we define $F_{U}$
to be the field of fractions of the $t$-adic completion $\wh{R}_{U}$ of the ring $R_{U} \subset F$
consisting of the rational functions on~$\XX$ that are regular at all points of $U$. Note that $\wh{R}_{U}$ is $I$-adically complete for the radical $I$ of the ideal generated by $t$ in $\wh{R}_{U}$.
The quotient
$  \wh{R}_{U}/I $ equals $k[U]$, the ring of regular functions on the integral, affine curve $U$.
For a (not necessarily closed)
 point $P$ of $X$, we let~$F_{P}$ denote the field of fractions of the complete local ring $\wh R_P:=\wh{\OO}_{\XX,P}$ of $\XX$ at $P$.

Let  $U \in \UU$ and let $P\in \PP$ be a closed point that is in the closure $\bar U$ of $U$ inside $X$ (recall that $\bar U$ is a regular curve).
The pair $(U,P)$ is a {\em branch} of $X$ at $P$ on $U$.
Let $\wh{R}_{U,P}$ denote the completion of the localization of $\wh{\OO}_{\XX,P}$ at the codimension~1 point that is associated to
the generic point of $\bar U$. This is a discrete valuation ring; let $F_{U,P}$ denote its fraction field.
There is an inclusion $F_{P} \subset F_{U,P}$ induced by the inclusion $\wh{R}_P\subset \wh{R}_{U,P}$.
There is also an inclusion $F_{U} \subset  F_{U,P}$, induced by the inclusion $\wh{R}_{U} \to \wh{R}_{U,P}$. (See \cite{admis}, beginning of Section~4.)

In the later sections, we will also need to consider residue fields and constant fields. For a point $P\in X$, we let $\kappa(P)$ denote its residue field. For an affine open $U\subseteq X$, we define $\kappa(U)=\OO (\bar U)$, the ring of functions on the closure of $U$; this is the field of constants of the irreducible component $\bar U$ of $X^{\operatorname{red}}$. Note that there is a natural inclusion of fields $\kappa(U)\rightarrow \kappa(P)$ when $P$ is on the closure of~$U$.

%----------------------------------------------------------------------------------------------------------------------------------------------------------------------------------------------------------------------------------------------------
\subsection{Obstructions to local-global principles over semi-global fields}\label{lgp}
%----------------------------------------------------------------------------------------------------------------------------------------------------------------------------------------------------------------------------------------------------

 Let $G$ be a linear algebraic group over a semi-global field~$F$, i.e., a smooth affine group scheme of finite type over $F$. In this subsection, we define various collections of overfields of~$F$, the associated local-global principles for $G$-torsors, and their obstructions, which are subsets of the Galois cohomology set $H^1(F,G)$.

Recall that a $G$-torsor $Z$ over $F$ is called {\em trivial} if it is isomorphic to $G$ as a $G$-space (with the action given by translation); equivalently, $Z$ has an $F$-point. Isomorphism classes of $G$-torsors over $F$ correspond bijectively to the elements of the pointed set $H^1(F,G)$, and under this identification, the trivial torsor corresponds to the trivial element in $H^1(F,G)$.

Given a collection of overfields $(F_i)_{i\in I}$ of $F$, one can consider the corresponding local-global principle for rational points on $G$-torsors: Must a $G$-torsor $Z$ which has a point over each $F_i$ also have an $F$-point? Equivalently, must an element in the kernel of $H^1(F,G)\rightarrow \prod_{i\in I}H^1(F_i,G)$ be the trivial class? Consequently, the kernel of the local-global map describes the obstruction to a local-global principle for rational points, for all $G$-torsors over $F$.

The first such obstruction set we consider is defined via discrete valuations, in analogy to the number field case.
Let $\Omega$ be the set of  discrete valuations of $F$. For $v \in \Omega$, let $F_v$ be the completion of $F$ at $v$.

 We define $$\Sha(F, G) = {\rm ker}(H^1(F, G) \to \prod_{v  \in \Omega} H^1(F_v, G));$$
 here the kernel is defined as the preimage of the trivial element.

The other obstruction sets we consider are defined using a normal crossings model~$\XX$ of the semi-global field~$F$. As above, let $X$ denote the closed fiber of~$\XX$. We then
define
$$\Sha_X(F, G) = {\rm ker}(H^1(F, G) \to \prod_{P \in X} H^1(F_P, G)),$$
where $P$ runs through all the points of $X$ (including generic points of components of $X$).

For $\PP$ a subset of the reduced closed fiber and corresponding $\UU$ as above, we let
$$\Sha_\PP(F, G) = {\rm ker}(H^1(F, G) \to \prod_{\zeta \in \UU \cup \PP } H^1(F_\zeta, G)).$$

Finally, let $\XX^{(1)}$ be the set of codimension one points of $\XX$, and let
$$\Sha_\XX(F, G) = {\rm ker}(H^1(F, G) \to \prod_{x \in \XX^{(1)}} H^1(F_x, G)).$$
Here for $x \in \XX^{(1)}$, $F_x$ denotes the fraction field of the complete local ring at~$x$ (which is the same as the completion of $F$ at the discrete valuation of $F$ given by $x$).

 There are several known containments among these obstruction sets. Namely,
 $$ \Sha_\PP(F, G) \subseteq \Sha_X(F, G) \subseteq \Sha(F, G)\subseteq \Sha_\XX(F, G).$$
 Here the first inclusion was shown in \cite[Corollary 5.9]{HHK3}, the second one is by \cite[Proposition 8.2]{HHK3}, and the final one is by definition.  One also has  $\cup_{\PP}  \Sha_\PP(F, G) =  \Sha_X(F, G)$ (\cite[Corollary 5.9]{HHK3}); the union is taken over all subsets $\PP$ that satisfy the conditions above.

 If $G$ is a reductive group over the scheme $\XX$ (rather than merely over $F$),
  then by \cite[Thm. 4.2(ii)]{CTPS1} one further has
 $$\cup_{\PP}  \Sha_\PP(F, G)  =  \Sha_X(F, G) = \Sha(F, G) = \Sha_\XX(F, G).$$
 (In loc.cit., it was assumed that the group is defined over the underlying discrete valuation ring $R$, but the proof of part~(ii) in fact only relies on $G$ being defined over $\XX$.)
 For tori, a strengthening of this is given at Theorem~\ref{allshaTequal} below.

The definitions here are given for a general linear algebraic group~$G$; in this manuscript we study these obstruction sets when $G$ is a torus. In that case the Galois cohomology sets are abelian groups, and the obstruction sets are subgroups.

Although completions with respect to discrete valuations are in closest analogy to classical local-global principles, the geometrically defined obstruction set $\Sha_{\PP}(F,G)$ is easier to compute explicitly (and as noted above, they are equal in interesting cases). In particular, we recall the following theorem.
 \begin{theorem}\cite[Cor. 3.6]{HHK3}\label{doublecosetG}
For any linear algebraic group $G$ over
$F$, we have a bijection
of pointed sets
$$ \Sha_\PP(F, G)  \simeq \left.  \prod _U G(F_U)  \middle\backslash \prod_{(U, P)} G(F_{U, P})
 \middle\slash \prod_P G(F_P) \right.  .$$
  Here the left and right hand side products run over $\UU$ and $\PP$, respectively, and the middle product runs over the branches $(U,P)$.
 For a commutative group $G$, this is an isomorphism of groups.
\end{theorem}

%===========================================================
\section{Reminders on $R$-equivalence and flasque tori}\label{reminders}
%===========================================================

This section contains reminders from \cite{CTS77}  and  \cite{CTS87} (to which we refer, including for the history of these results).

Let $k$ be a field, and let $k_{s}$ be a separable algebraic closure of $k$.
A {\em $k$-torus} is an algebraic group $T$ over $k$ such that $T\times_{k}k_{s}$
is isomorphic to a product of copies of ${\mathbb G}_{m,k_{s}}$.
Its {\em character group} is the group of homomorphisms of $k_s$-algebraic groups $\hat{T}=\Hom_{k_{s}\mathrm{\mhyphen gp}}(T\times_{k}k_{s},{\mathbb
G}_{m,k_{s}})$.
It is a $\Gal(k_{s}/k)$-lattice, i.e., a free finitely generated abelian group with a continuous,
discrete action of $\Gal(k_{s}/k)$. This $\Gal(k_{s}/k)$-lattice determines the $k$-torus.
The $k$-torus is said to be {\em split} by a  Galois extension $\ell /k$ if
$T\times_{k}\ell$ is $\ell$-isomorphic to a product of copies of ${\mathbb G}_{m,\ell}$.

A {\em quasitrivial} $k$-torus $Q$ is a $k$-torus which is $k$-isomorphic to
a finite product  $\prod_{i}R_{k_{i}/k}{\mathbb G}_{m}$, for
 finite separable field extensions $k_i$ of $k$.
Such a $k$-torus $Q$ is an open subset of  $\prod_{i}R_{k_{i}/k}{\mathbb G}_{a}$,
which is $k$-isomorphic to an affine space $\A^n_{k}$. In particular, $Q$ is a rational linear algebraic group.
Shapiro's Lemma and Hilbert's Theorem~90 imply that the
Galois cohomology group $H^1(k,Q)$ of a quasitrivial $k$-torus $Q$ is trivial.
These properties are stable under field extensions of $k$.

A {\em flasque} $k$-torus $S$ is a $k$-torus such that $H^1(H,\Hom_{\Z}(\hat{S},\Z))=0$
for all closed subgroups $H$ of $\Gal(k_{s}/k)$.

Given any $k$-torus $T$, there exists an exact sequence of $k$-tori
$$ 1 \to S \to Q \to T \to 1$$
with $Q$ quasitrivial and $S$ flasque (see \cite{CTS77}, \cite{CTS87}, where the idea is attributed to Voskresenski\u{\i}).
The short exact sequence is called a  {\em flasque resolution of $T$}.
(The torus $S$ in such a sequence is well defined up to taking a product with a quasitrivial torus.)

By taking Galois cohomology, or \'etale cohomology, such a flasque resolution
induces an exact sequence
$$ Q(k) \to T(k) \to H^1(k,S) \to 0$$
since $H^1(k,Q)=0$, as recalled above.

Another way of interpreting $H^1(k,S)$ is in terms of $\rm R$-equivalence.
Let $Z$ be a $k$-variety. Two $k$-points $P,Q \in Z(k)$
 are called {\em elementary linked} if there exist an open subset~$U$ of the affine line $\A^1_{k}$ and a $k$-morphism
$f : U \to Z$
such that $P$ and $Q$ are in $f(U(k))$. One then defines {\em $\rm R$-equivalence}
on $Z(k)$ as the equivalence relation generated by this relation.
If $Z=G$ is an algebraic group over $k$, the set of $k$-points that are $\rm R$-equivalent to $1 \in G(k)$
is a normal subgroup. The quotient is denoted by $G(k)/\rm R$.

Now if $Q$ is a quasitrivial torus, then since $Q$ is an open subset of affine space,
$Q(k)/{\rm R}=1$. For a flasque resolution $1\rightarrow S\rightarrow Q\rightarrow T\rightarrow 1$, we thus have  an isomorphism
$T(k)/{\rm Im}(Q(k)) \simeq H^1(k,S)$ and a surjection $T(k)/{\rm Im}(Q(k)) \to T(k)/\rm R$.
In fact (see \cite[Theorem~3.1]{CTS87}) the latter map is an isomorphism and we have
 $$T(k)/{\rm R} \simeq H^1(k,S).$$

For any noetherian scheme $Z$, one defines tori over $Z$, quasitrivial tori over $Z$, flasque tori over $Z$, and flasque resolutions of tori over $Z$ in
an analogous way; the latter always exist. These notions are
functorially contravariant with respect
to any morphism $Y \to Z$ (e.g., see \cite[Prop.~1.4]{CTS87} for the
pullback of flasque resolutions).

In this paper, a torus $T$ over $X$ is by definition isotrivial.
Namely, there exists a finite, \'etale, surjective map $Y \to X$
such that $T\times_XY$ is a split torus over $Y$, i.e.. it is isomorphic to
a power $\G^r_{m,Y}$ of the multiplicative group $\G_{m,Y}$. 
We refer to \cite[\S 0]{CTS87} for basic definitions and properties of
tori over a scheme.

\medskip

We will use the following basic properties:

\begin{properties}\label{properties_flasque}
\renewcommand{\theenumi}{\alph{enumi}}
\renewcommand{\labelenumi}{(\alph{enumi})}
\begin{enumerate}
\item \label{torus coho}
Let $A$ be a regular local ring with
fraction field $L$ and residue field $k$, and let $T$ be an $A$-torus.
Then the natural restriction map of \'etale cohomology groups
$H^1_{\text{\'et}}(A,T) \to H^1(L,T)$ is injective (by \cite[Theorem~4.1(i)]{CTS87}).
If $A$ is complete then the restriction map $H^1_{\text{\'et}}(A,T) \to
H^1(k,T)$ is an isomorphism (\cite[Prop. 8.1]{Dema}, \cite[Chap. III, Remark 3.11]{Milne}).
\item \label{flasque coho surj}
\cite[Theorem~2.2(i)]{CTS87}  If $S$ is a flasque torus over a regular
connected scheme $Z$, then for any open set $U \subset Z$,
the restriction map $H^1_{\text{\'et}}(Z,S) \to H^1_{\text{\'et}}(U,S)$
is surjective. Moreover, the restriction map
$H^1_{\text{\'et}}(Z,S) \to H^1(L,S)$  is also surjective, where $L$ is
the function field of $Z$.  (This condition motivates the
terminology ``flasque'', in analogy with the term ``flasque sheaf''.)
\item \label{aff flasque coho iso}
\cite[Corollary~2.6]{CTS87} If $k$ is a field, $S$ is a flasque
$k$-torus, and if $U$ is a non-empty open subset of $\A^n_k$, the
composition $H^1(k,S) \to H^1_{\text{\'et}}(\A^n_k,S) \to
H^1_{\text{\'et}}(U,S)$ is an isomorphism.
\end{enumerate}
\end{properties}

Thus if $A$ is a regular local ring with fraction field $L$ and residue
field $k$, and $S$ is a flasque torus over $A$, then there is a
{\em specialization} map $H^1(L,S) \to H^1(k,S)$ that is given by
\[H^1(L,S) \overset{\simeq}{\longleftarrow} H^1_{\text{\'et}}(A,S) \to
H^1(k,S),\]
and which is an isomorphism if $A$ is complete.

In particular, there is a specialization map in the case of a discrete
valuation ring $A$, such as the local ring at a closed point on a regular curve with function field $L$.  In the case that $A$ is a complete regular local
ring of dimension two with residue field $k$, and $\pi$ is a regular
prime of $A$ (i.e., $A/(\pi)$ is regular),
after Lemma~\ref{residue} below
we also give a specialization map $H^1(L_\pi,S)\rightarrow H^1(k,S)$
which is compatible with the specialization map $H^1(L,S) \to H^1(k,S)$ (here $L_\pi$ denotes the completion of $L$ with respect to $\pi$).

%===========================================================
\section{The case of $F$-tori}
%===========================================================

In this section, we give a new double coset description of $\Sha_{\PP}(F,T)$ when $T$ is an $F$-torus. This description can be stated using a flasque resolution, or using ${\rm R}$-equivalence.

 \begin{theorem}\label{H1Scosets}
 \label{sharequiv}
Let $F$ be a semi-global field, let $\XX$ be a normal crossings model of~$F$ with closed fiber~$X$, and let $\PP\subset X^{\operatorname{red}}$ and $\UU$ be as in Section~\ref{setup}.
Let $T$ be an $F$-torus,
 and let $1 \to S \to Q \to T \to 1$ be a flasque resolution of $T$ over $F$.

\renewcommand{\theenumi}{\alph{enumi}}
\renewcommand{\labelenumi}{(\alph{enumi})}

 \begin{enumerate}
\item\label{F-tori_a} There is an isomorphism of abelian groups
$$ \Sha_\PP(F, T)  \simeq \left. \prod _U H^1(F_U,S) \middle\backslash \prod_{(U, P)} H^1(F_{U, P},S)
 \middle\slash  \prod_P H^1(F_P,S)  \right. .$$

 \item\label{F-tori_b} There is an isomorphism of abelian groups
  $$   \Sha_\PP(F, T)\simeq\left. \prod _U T(F_U)/{\rm R}  \middle\backslash \prod_{(U, P)} T(F_{U, P})/{\rm R}
 \middle\slash \prod_P  T(F_P)/{\rm R}  \right. ,$$
where ${\rm R}$ denotes ${\rm R}$-equivalence.
\end{enumerate}
 Here in each double coset decomposition, the left and right hand side products run over $\UU$ and $\PP$, respectively, and the middle product runs over the branches $(U,P)$.
  \end{theorem}

 \begin{proof}
Since $Q$ is quasitrivial,  $H^1(F,Q)=0$. In particular, $\Sha_\PP(F,Q)=0$.
By Theorem~\ref{doublecosetG} applied to $Q$, this implies that the natural map
$\prod\limits_U Q(F_U)\times \prod\limits_P Q(F_P) \to \prod\limits_{(U,P)}Q(F_{U,P})$
is onto. Applying cohomology and using the vanishing of $H^1(L,Q)$ for any overfield $L$ of $F$,
we
 obtain the commutative diagram of exact sequences
 $$
 \resizebox{\displaywidth}{!}{%
\hspace{-2em}\xymatrix{
 \prod\limits_U Q(F_U)\times\prod\limits_P Q(F_P)\ar[r] \ar[d]&\prod\limits_{(U,P)}Q(F_{U,P})
\ar[r]
\ar[d]& 1
\\
\prod\limits_U T(F_U)\times\prod\limits_P T(F_P) \ar[r]\ar[d] &\prod\limits_{(U,P)}T(F_{U,P})\ar[r]\ar[d]& \prod \limits_U T(F_U)\backslash \prod\limits_{(U, P)} T(F_{U, P})
\slash \prod\limits_P  T(F_P)  \ar[d] \ar[r]  &  1 \\
\prod\limits_U H^1(F_U,S)\times \prod\limits_P H^1(F_P,S)  \ar[r] \ar[d] & \prod\limits_{(U,P)}H^1(F_{U,P},S)\ar[r]\ar[d]&  \prod\limits _U H^1(F_U,S) \backslash \prod\limits_{(U, P)} H^1(F_{U, P},S)
\slash \prod\limits_P  H^1(F_P,S) \ar[d]\ar[r]  & 1  \\
  1 &1 & 1
}}
$$

\bigskip
A diagram chase gives that  the map
$$ \prod \limits_U T(F_U) \backslash \prod\limits_{(U, P)} T(F_{U, P})
\slash  \prod\limits_P  T(F_P) \to
\prod\limits _U H^1(F_U,S)  \backslash \prod\limits_{(U, P)} H^1(F_{U, P},S)
\slash \prod\limits_P  H^1(F_P,S)  $$ is an isomorphism.
By Theorem \ref{doublecosetG} applied to $T$
there is an isomorphism
  $$ \Sha_\PP(F, T)  \simeq\left.   \prod _U T(F_U)\middle\backslash \prod_{(U, P)} T(F_{U, P})
 \middle\slash  \prod_P  T(F_P)  \right. .$$
This proves~(\ref{F-tori_a}). Then~(\ref{F-tori_b}) follows from~(\ref{F-tori_a}) and the functorial isomorphisms $T(L)/{\rm R} \simeq H^1(L,S)$
for any field extension  $L$ of $F$.
  \end{proof}

As a first application, we obtain a finiteness result.

 \begin{theorem}\label{Ftorus-finite}
 Let $k$ be either a field that is finitely generated
   over ${\mathbb Q}$ or a local field of characteristic zero, and let $K=k((t))$. Let $F$ be a semi-global field over $K$ with normal crossings model $\XX$ and closed fiber~$X$.
Let $\PP \subset X^{\operatorname{red}}$ and $\UU$ be as in Section~\ref{setup}.
 Then for any $F$-torus $T$, $\Sha_\PP(F, T)$ is finite.
  \end{theorem}
 \begin{proof}
 We claim that for each $(U,P)$ and any flasque torus~$S$, $H^1(F_{U,P},S)$ is finite. Since there are only finitely many pairs $(U, P)$,
 Theorem~\ref{H1Scosets} will then imply the finiteness of $\Sha_\PP(F, T)$.

Each field $F_{U,P}$ is  of the shape $\ell((u))((t))$, with $\ell$ a finite field extension of $k$. By \cite[Thm. 3.2]{CGP}, if $L$ is a field of characteristic zero such that $H^1(L,S)$ is finite for {\em every} flasque torus $S$ over $L$, then $H^1(L((t)),S)$ is finite for every flasque torus $S$ over $L((t))$.  So to prove our claim, it suffices to show that $H^1(\ell,S)$ is finite for every flasque torus~$S$. By assumption, $\ell$ is either finitely generated over ${\mathbb Q}$ or local. Finiteness in the former case follows from \cite[Thm.~1, p.~192]{CTS77}. Finiteness in the latter case is well known for any connected linear algebraic group (see e.g. \cite{PR94}, Theorem~6.14).
  \end{proof}

 See \cite[Thm. 3.4]{CGP} and the remark following it
for further examples of fields $L$ with the property that $H^1(L,S)$ is finite
 for every flasque torus $S$.

%==========================================================
 \section{The case of $\XX$-tori}\label{sec4}
%===========================================================
In this section, the torus under consideration is defined over a normal crossings model $\XX$ of the semi-global field. We first show that the different obstruction sets all coincide in that case. In the second part we give a description of the terms occurring in the double coset formula Theorem~\ref{H1Scosets}.

%--------------------------------------------------------------------------------------------------------
 \subsection{Comparing various Tate-Shafarevich groups}\label{varioushas}
%--------------------------------------------------------------------------------------------------------

In this subsection we prove a local factorization result for tori defined over a normal crossings model of the semi-global field. Using this, we show that the different obstruction sets defined in Section~\ref{lgp} all coincide.

 \begin{lemma}\label{residue}
 Let $A$ be a complete regular local ring of dimension~$2$ with residue field $k$ and field of fractions $L$. Let $L_\pi$ be the completion of $L$ at a regular prime $\pi$ of $A$.  Moreover, let $T$ be a
 torus over $A$, and let $1\rightarrow S\rightarrow Q\rightarrow T\rightarrow 1$ be a flasque resolution over $A$.
Let $\theta:H^1(L_\pi,S)\to H^1(k,S)$ be the composition of the specialization maps
$H^1(L_\pi,S) \to H^1(k(\pi),S) \to H^1(k,S)$.  Then $\theta$ is an isomorphism and the diagram
 $$\xymatrix{
T(L_\pi) \ar[r] &H^1(L_\pi,S)\ar[dd]^\theta\\
 T(A)\ar[u] \ar[d]         &\\
 T(k)\ar[r]&H^1(k,S)}
$$
 commutes. Here the horizontal maps are the coboundary maps in the cohomology sequence coming from the flasque resolution.
 \end{lemma}

 \begin{proof}
The two specialization maps $H^1(L_\pi,S) \to H^1(k(\pi),S)$ and $H^1(k(\pi),S) \to H^1(k,S)$ are isomorphisms by the completeness of the discrete valuation rings $\wh A_{(\pi)}$ and $A/(\pi)$, where the former ring is the completion of the localization of $A$ at $\pi$; see the comment after Properties~\ref{properties_flasque}.  Thus $\theta$ is an isomorphism.  For the commutativity of the diagram, recall that the above two specialization maps are respectively given by
\[H^1(L_\pi,S) \overset{\simeq}{\longleftarrow}
 H^1_{\text{\'et}}(\wh A_{(\pi)},S)\to H^1(k(\pi),S), \
 \ H^1(k(\pi),S) \overset{\simeq}{\longleftarrow}
  H^1_{\text{\'et}}(A/(\pi),S)\to H^1(k,S).\]
The commutativity now easily follows using that the compositions $T(A)\rightarrow
T(\wh A_{(\pi)})\rightarrow T( k(\pi)))$ and $T(A)\rightarrow T(A/(\pi))\rightarrow T(k(\pi))$ define the same maps and from the functoriality of the coboundary map $T(\cdot)\rightarrow H^1(\cdot ,S)$.
\end{proof}

The isomorphism $\theta$, which is a composition of two specialization maps, is compatible with the
specialization map $H^1(L,S) \to H^1(k,S)$ described after Properties~\ref{properties_flasque}.  That is,
the latter map is the same as the composition $H^1(L,S)\rightarrow H^1(L_\pi,S)\overset{\theta}{\rightarrow} H^1(k,S)$, since the restriction maps to the residue fields are isomorphisms; see Properties~\ref{properties_flasque}(\ref{torus coho}).
We thus also call $\theta$ the {\em specialization map} from  $H^1(L_\pi,S)$ to $H^1(k,S)$.

Now let $F$ be a semi-global field with normal crossings model $\XX$ and closed fiber $X$. Let $\PP$ and $\UU$ be as in Section~\ref{setup}. We next use the specialization map to prove a factorization lemma.
 \begin{lemma}\label{factorization}
Let $T$ be a
torus
over $\XX$.
Let $1 \rightarrow S \rightarrow Q\overset{\varphi}{\rightarrow} T\rightarrow 1$ be a flasque resolution of $T$ over $\XX$. Let $U$ be a nonempty affine open subset of an irreducible component of the reduced closed fiber $X^{\operatorname{red}}$ of $\XX$ such that $U$ does not meet any other components of $X^{\operatorname{red}}$. Then for any $P$ in the complement $\bar U\setminus U$ of $U$ in its closure, $T(F_{U,P})=T(\wh R_P)\varphi(Q(F_{U,P}))$.
\end{lemma}
\begin{proof}
Consider the following diagram
$$\xymatrix{
Q(F_{U,P})\ar[r]^\varphi&T(F_{U,P})\ar[r]& H^1(F_{U,P},S)\ar[r]\ar[dd]^\theta&H^1(F_{U,P},Q)=0\\
&T(\wh R_P)\ar[u]\ar[d]\\
Q(k(P))\ar[r]&T(k(P))\ar[r]\ar[d]& H^1(k(P),S)\ar[r]&H^1(k(P),Q)=0\\
&1}
$$
where $\theta$ is the isomorphism given by Lemma~\ref{residue}. The two cohomology groups on the right vanish since $Q$ is quasi-trivial.
The assertion now follows from an easy diagram chase.
\end{proof}

  \begin{prop}  \label{FU} Let $F$ be a semi-global field with normal crossings model $\XX$, and let $T$ be a  torus  over $\XX$.
Let $U$ be a nonempty affine open subset of an irreducible component of the reduced closed fiber $X^{\operatorname{red}}$ of $\XX$ such that $U$ does not meet any other components of $X^{\operatorname{red}}$.
Let  $P \in U$ be a closed point  and $V = U \setminus \{ P \}$.
Then $\operatorname{ker}(H^1(F_U, T) \to H^1(F_V, T) \times  H^1(F_P, T))$ is trivial.
\end{prop}

\begin{proof}
Let $1 \to S \to Q \to T  \to 1$ be a flasque resolution of $T$ over the ring $\wh R_P$ defined in Section~\ref{setup} (we continue to use $T$ to denote the restriction of $T$ over $\wh R_P$).
By Lemma~\ref{factorization} applied to $P$ and $V$, any element of $T(F_{V, P})$
is  the product of an element of $T(\wh R_P)  \subset T(F_{P})$ and an element in
the image of $Q(F_{V,P})$. Since $Q$ is a quasi-trivial $F$-torus, $Q$ is rational, and thus
$Q(F_{V,P}) = Q(F_{V}) Q(F_{P})$
by \cite[Prop. 3.9 and Corollary 3.15]{HHK5}.
Thus $T(F_{V, P}) = T(F_V)T(F_P)$.
By \cite[Theorem 2.13 and Proposition 3.9]{HHK5}, this concludes the proof.
\end{proof}

The next theorem shows that in the case of tori, it is sufficient to consider $ \Sha_\PP(F, T)$ for some $\PP$ (rather than taking the union over all possible $\PP$).
\begin{theorem}\label{allshaTequal}
\label{shas} Let $F$ be a semi-global field with normal crossings model~$\XX$ and let $\PP $ be as in Section~\ref{setup}.
Let $T$ be a  torus  over $\XX$.
Then the subgroups  $ \Sha_\PP(F, T)$,  $\Sha_X(F, T)$,
 $\Sha(F, T) $  and $ \Sha_\XX(F, T)$  of $H^1(F,T)$ all coincide.
\end{theorem}

\begin{proof}
By the chain of equalities given just before Theorem~\ref{doublecosetG}, it suffices to show that $\Sha_X(F, T)\subseteq \Sha_\PP(F, T)$.
So let $\zeta \in \Sha_X(F, T)$.  Then $\zeta \otimes F_P$ is trivial for every point $P \in X$ by definition (including generic points of components of~$X$).
If $U \in \UU$, then by \cite[Proposition 5.8]{HHK3},
  there exists a
 non-empty open subset $V$ of $U$ such that $\zeta \otimes F_V$ is trivial.
 Since $U \setminus V$ is a finite set, by  Proposition \ref{FU}, $\zeta  \otimes F_U$ is trivial.
 Hence $\zeta \in \Sha_\PP(F, T)$.
\end{proof}

Combining this with Theorem~\ref{Ftorus-finite}, we immediately obtain:
  \begin{theorem}\label{finite}
  Let $k$ be either a finitely generated field
   over ${\mathbb Q}$ or a local field of characteristic zero, and let $K=k((t))$.
 Let $F$ be a semi-global field over~$K$ with normal crossings model~$\XX$, and let $\PP$ be as in Section~\ref{setup}.
   Let $T$ be a torus over $\XX$.
Then the groups $ \Sha_\PP(F, T)$,  $\Sha_X(F, T)$,
 $\Sha(F, T) $  and $ \Sha_\XX(F, T)$ are all finite.
 \end{theorem}

 \begin{proof} By Theorem \ref{Ftorus-finite}, $\Sha_\PP(F, T)$ is finite.
 The result then follows from Theorem \ref{allshaTequal}.
 \end{proof}

See Example~\ref{infinite} for an example where $\Sha(F,T)$ is infinite.

%--------------------------------------------------------------------------------------------------------
\subsection{On the value of $\Sha_\PP(F, T)$ when $T$ is an $\XX$-torus}
%--------------------------------------------------------------------------------------------------------

For a torus $T$ over a semi-global field~$F$, Theorem~\ref{H1Scosets} gives a formula for $\Sha_\PP(F, T)$ in terms of a flasque resolution of $T$.

If $T$ is an $\XX$-torus, and $1\rightarrow S\rightarrow Q\rightarrow T\rightarrow 1$ is a flasque resolution of $T$ over $\XX$,
pullback
induces a flasque resolution of $T$ over $F$.
In that situation, we now analyze the maps
$H^1(F_{P},S) \to H^1(F_{U,P},S)$ and $H^1(F_{U},S) \to H^1(F_{U,P},S)$
which were relied on in the formula in Theorem~\ref{H1Scosets}(\ref{F-tori_a}).

In what follows, we  abuse notation:
For $S$ a $Z$-torus and $Y \to Z$ a morphism of schemes, we let $H^1(Y,S)$ denote the \'etale cohomology group $H^1_{\et}(Y,S_{Y})$, where $S_{Y}$ is the $Y$-torus $S \times_{Z}Y$.
For an affine scheme $Z=\Spec A$,
we write $H^1(Z,S)=H^1(A,S)$.

Recall that for a point $P\in X$, $\kappa(P)$ denotes its residue field.

\begin{prop}\label{value}  Let $F$ be a semi-global field with normal crossings model $\XX$, and let $\PP$ and $\UU$ be as in Section~\ref{setup}.
Let $S$ be a flasque torus over $\XX$.
\renewcommand{\theenumi}{\alph{enumi}}
\renewcommand{\labelenumi}{(\alph{enumi})}
\begin{enumerate}
\item\label{value_a} For each point $P\in \PP$, the specialization map $H^1(F_{P},S) \to  H^1(\kappa(P),S)$ is an isomorphism.
 \item\label{value_b} For each branch $(U,P)$,
 the specialization map $\theta:H^1(F_{U,P},S)
\to H^1(\kappa(P),S)$ from Lemma~\ref{residue} is an isomorphism that
is compatible with the isomorphism in (\ref{value_a}), in the sense
 that we have a commutative diagram
 \[\xymatrix{
  H^1(F_P, S) \ar[r]^\sim \ar[d] & H^1(\kappa(P), S) \ar[d]_= \\
  H^1(F_{U, P}, S) \ar[r]^\sim & H^1(\kappa(P), S).
 }\]
\item\label{value_c} For each branch $(U,P)$,
consider the map
\[H^1(F_{U},S) \to H^1(F_{U,P},S) \overset{\sim}\to H^1(\kappa(P),S)\]
induced by the inclusion $F_{U} \subset F_{U,P}$.  For a fixed $U \in \UU$,  the
images of the following maps coincide:
 \begin{itemize}
 \item $H^1(F_{U},S) \to \prod_{(U, P)}H^1(F_{U,P},S) \overset{\sim}\to \prod_{(U, P)}
 H^1(\kappa(P),S)$.
 \item the product of the restriction maps $H^1(\bar U,S) \to \prod_{(U, P)} H^1(\kappa(P),S)$.
 \item the product of the specialization maps $H^1(k(U), S) \to \prod_{(U,
 P)} H^1(\kappa(P), S)$.
  \end{itemize}
 Here $\bar U$ denotes the closure of
 $U$, and the products are taken over all branches $(U,P)$ at points $P \in \PP$ on $\bar U$.
 \end{enumerate}
 \end{prop}

  \begin{proof} Part~(\ref{value_a}) follows from the completeness of the local ring $\wh R_P$, by Properties~\ref{properties_flasque}.
 For part~(\ref{value_b}), it was shown in Lemma~\ref{residue} that $\theta$ is an isomorphism; and the commutativity of the diagram follows from the compatibility stated just after the proof of that lemma (taking $L=F_P$ and $L_{\pi} = F_{U,P}$).  For part~(\ref{value_c}), consider the commutative diagram

 $$\begin{array}{ccccccc}
  && H^1(F_{U},S)  &   \lra&  \prod_{(U, P)}H^1(F_{U,P},S) & &\\
  &&\twoheaduparrow&&\uparrow\simeq  & \simeq\searrow^\theta &\\
  &&H^1(\wh{R}_{U},S)  & \lra&  \prod_{(U, P)}H^1(\wh{R}_{U,P},S) &  \lra & \prod_{(U, P)}H^1(\kappa(P),S)    \\
  &&\downarrow \simeq& &\downarrow \simeq &      &||   \\
   H^1(\bar U,S) & \twoheadrightarrow &
  H^1(k[U],S) &  \lra & \prod_{(U, P)}H^1(\widehat{k(U)}_{P},S)   & \stackrel{\sim} \lra & \prod_{(U, P)}H^1(\kappa(P),S) \\
  && || && \uparrow&&\\
  && H^1(U,S) & \twoheadrightarrow & H^1(k(U),S),&&
   \end{array}$$
where $\widehat{k(U)}_{P}$ is the completion of $k(U) = k(\bar U)$ at the discrete valuation associated to $P$.
Here the two vertical maps from $\prod_{(U, P)}H^1(\wh{R}_{U,P},S)$ are isomorphisms by Properties~\ref{properties_flasque}.  For each $P$, the composition $H^1(F_{U,P},S) \to H^1(\widehat{k(U)}_{P},S)$
and the map $H^1(\widehat{k(U)}_{P},S) \to H^1(\kappa(P),S)$ are both
specialization isomorphisms.  The map
$H^1(\wh{R}_{U},S) \to H^1(k[U],S)$ is also an isomorphism, because
$\wh{R}_{U}$ is complete with respect to an ideal $I$ such that $\wh{R}_{U}/I=k[U]$ (\cite{Strano}, Theorem~1).  Three other maps in the above diagram are surjections, as indicated, by  Properties~\ref{properties_flasque}(\ref{flasque coho surj}).  It then follows from this diagram that the three maps considered in the statement of~(\ref{value_c}) each have the same image as the map $H^1(k[U],S) \to \prod_{(U, P)}H^1(\kappa(P),S)$.
\end{proof}

The following corollary requires the flasque torus $S$ to be defined over the underlying discrete valuation ring~$R$ rather than just over $\XX$.
A group scheme $S$ over the complete discrete valuation ring $R$ induces $S_{\XX}$ over $\XX$ by pullback,
as well as its restrictions to subschemes of $\XX$.
We will continue to drop the subscripts if these group schemes appear as coefficients in cohomology groups.

Recall that for an affine open $U\subseteq X$, we define $\kappa(U)=\OO (\bar U)$, the ring of functions on the
closure of $U$, and that $\kappa(U)$ naturally embeds into $\kappa(P)$ when $P$ is a point on the closure of~$U$.

\begin{cor}\label{corvalue}  Let $K$ be a complete discretely valued field with ring of integers~$R$, let $F$ be a semi-global field over~$K$ with normal crossings model~$\XX$,  and let $\PP$ and $\UU$ be as in Section~\ref{setup}.  Let $S$ be a flasque $R$-torus.
\renewcommand{\theenumi}{\alph{enumi}}
\renewcommand{\labelenumi}{(\alph{enumi})}
\begin{enumerate}
\item\label{cor_branch} Let $(U,P)$ be a branch. If  $\kappa(U)=\kappa(P)$, then the map
 $H^1(\bar U ,S) \to H^1(\kappa(P),S)$ is surjective,
 hence so is the map $H^1(F_{U},S) \to H^1(\kappa(P),S)$.

 \item\label{cor_component}
  Assume that each component  of the closed fiber $X$ of $\XX$ is $k$-isomorphic to $\P^1_{k}$.
 Let $U\in \UU$, and let $P_{1} , \dots, P_{n}$
be the closed points of $\PP$ which lie on $\bar U$. Then
 the image of the map
 $H^1(F_{U},S)  \to \prod_{j=1}^n H^1(\kappa(P_{j}),S)$
 coincides with the image of the diagonal map
 $$H^1(k,S) \to \prod_{j=1}^n H^1(\kappa(P_{j}),S).$$
 \end{enumerate}
  \end{cor}
  \begin{proof}
  (\ref{cor_branch}):
  We have an inclusion morphism $i:P \to \bar U$, and by the hypothesis  we also have a structure morphism $j:\bar U \to \Spec(\kappa(P)) = P$.
The composition $ji$ is the identity.  Taking cohomology, we obtain
$i^*:H^1(\bar U,S) \to H^1(P,S)$ and
$j^*:H^1(P,S) \to H^1(\bar U,j^*(S)) = H^1(\bar U,S)$.  Here the last
equality follows from the fact that $S$ is defined over $\bar U$ via pullback
from $R$, together with the fact that the map $\bar U \to \Spec(k) \subset
\Spec(R)$ factors through $P$ by hypothesis.
Since $ji$ is the identity, the composition
$$H^1(\kappa(P),S) \to H^1(\bar U,S) \to H^1(\kappa(P),S),$$
is also the identity map, yielding the first assertion.
The second assertion follows from the
first one by the last assertion  of Proposition~\ref{value}~(c).

  (\ref{cor_component}):  By Proposition~\ref{value}(\ref{value_c}), the image $H^1(F_{U},S)  \to \prod_{j=1}^n H^1(\kappa(P_{j}),S)$
  coincides with the image of $H^1(\bar U,S) \to \prod_{j=1}^n H^1(\kappa(P_{j}),S)$.
  If $k$ is finite, the statement is clear since all the $H^1(\kappa(P_{j}),S)$ vanish (use the Lang isogeny, \cite{lang}, Theorem~2).
  If $k$ is not finite, let $V \subset \bar U$ be an open subset
   isomorphic to $\A^1_{k}$ containing all points $P_{j}$.
  Then the map $H^1(\bar U,S) \to \prod_{j=1}^n H^1(\kappa(P_{j}),S)$ factors through $H^1(V,S)=H^1({\mathbb A}^1_k,S)=H^1(k,S)$ by Properties~\ref{properties_flasque}(\ref{aff flasque coho iso}),  and the assertion follows.
 \end{proof}

%==================================================
\section{Decorated Graphs and Contraction Results}\label{graphs}
%==================================================

For a semi-global field $F$ and an $F$-torus $T$, we will obtain results about $\Sha(F,T)$ by considering a graph associated to the closed fiber of a normal crossings model of $F$.  In this section, we will establish the graph theoretic prerequisites for this study.

%--------------------------------------------------------------------------------------------------------
\subsection{Decorated Graphs}
%--------------------------------------------------------------------------------------------------------

When we use the term {\em graph}, we will allow multiple edges between vertices as well as loops. It will be convenient notationally to consider every edge as having two distinct ``ends,'' or ``halves,'' each of which will be attached to a vertex (possibly the same one).  All graphs will be assumed finite.

Given a graph $\Gamma$ with vertex set $V$ and edge set $E$, and an edge $e$  connecting distinct vertices $x$ and $y$, we can form a new graph $\Gamma/e$, called the {\em contraction of $\Gamma$ along $e$}, by removing the edge $e$ and identifying the vertices $x$ and $y$ (denoting the resulting vertex by $[xy]$).
More formally, $\Gamma/e$ is the graph with vertex set $V' = (V \setminus \{x, y\}) \cup \big\{[xy]\big\}$ and edge set $E' = E \setminus \{e\}$,
where an edge $e' \in E'$ is adjacent to $v' \in V' \setminus \big\{[xy]\big\}$ if $e'$ is adjacent to $v'$ in $\Gamma$, and where $e'$ is adjacent to $[xy]$ in $\Gamma'$ if $e'$ is adjacent to either $x$ or $y$ in $\Gamma$.

Let $\Gamma$ be a graph. A {\em coefficient system} $A$ on $\Gamma$ is a rule that associates abelian groups $A_e$ and $A_x$ to each edge $e$ and vertex $x$, and associates a homomorphism $A_\alpha: A_x \to A_e$ to each half edge $\alpha$ of $e$ that is attached to $x$.  We refer to a pair $(\Gamma, A)$ consisting of a graph $\Gamma$ and a coefficient system $A$ on $\Gamma$ as a {\em  decorated graph}.

If $A, B$ are coefficient systems on the same graph $\Gamma$, we define a {\em morphism} $A \to B$ to be a collection of group homomorphisms $A_x \to B_x$ and $A_e \to B_e$ which commute with the maps $A_\alpha$. Further, we can form the kernel and cokernel of a morphism componentwise, and obtain an abelian category -- these are, in fact, just the abelian categories of morphisms from a diagram category coming from the graph $\Gamma$ to the category of abelian groups.

%--------------------------------------------------------------------------------------------------------
\subsection{Cohomology of Decorated Graphs}
%--------------------------------------------------------------------------------------------------------

Given a coefficient system $A$ on a graph $\Gamma$, we write $\mathcal C(\Gamma, A)$ for the cochain complex
\[\xymatrix @C=2cm @R=.2cm{
\mathcal C^0(\Gamma, A) \ar@{=}[d] \ar[r]^{\sum_\alpha A_\alpha} & \mathcal C^1(\Gamma, A) \ar@{=}[d]  \\
\oplus_{x} A_x & \oplus_e A_e  }\]
concentrated in degrees $0$ and $1$. Note that this gives a faithful (though not necessarily full) exact functor from the abelian category of coefficient systems on $\Gamma$ to the abelian category of abelian complexes. We define $H^i(\Gamma, A) = H^i(\mathcal C(\Gamma, A))$.

\begin{lemma}\label{ses}
Suppose that we have a graph $\Gamma$ and a short exact sequence of coefficient systems $0 \to A \to B \to C \to 0$ on $\Gamma$. Then we obtain a $6$-term exact sequence of cohomology:
\[ 0 \to H^0(\Gamma, A) \to H^0(\Gamma, B) \to H^0(\Gamma, C) \to H^1(\Gamma, A) \to H^1(\Gamma, B) \to H^1(\Gamma, C) \to 0 \]
\end{lemma}
\begin{proof}
By the exactness of the above functor,  $0 \to A \to B \to C \to 0$ induces an exact sequence of complexes $0 \to \mathcal C(\Gamma, A) \to \mathcal C(\Gamma, B) \to \mathcal C(\Gamma, C) \to 0$ and thus a long exact sequence of cohomology groups (where the higher cohomology groups vanish since each complex is concentrated in degrees~$0$ and~$1$).
\end{proof}

%----------------------------------------------------------------------------------------------------------------------------------------------------------------------------------------------------------------------------------------------------
\subsection{Contraction of coefficients}\label{contraction-subsection}
%----------------------------------------------------------------------------------------------------------------------------------------------------------------------------------------------------------------------------------------------------

Suppose we are given a decorated graph $(\Gamma, A)$, and $e$ is an edge with half edges $\alpha$ and $\beta$ attached to vertices $x$ and $y$ respectively. We say that the half edge $\alpha$ is {\em redundant} if $x \neq y$ and $A_\alpha: A_x \to A_e$ is an isomorphism. In this case, we may define a new coefficient system $A/\alpha$ on the graph $\Gamma/e$ as follows. As before, write $V'$ and $E'$ for the vertices and edges of $\Gamma/e$. We set
\[ (A/\alpha)_v = A_v, \text{ for } v \in V' \setminus\big\{[xy]\big\}, \ \ \ (A/\alpha)_{[xy]} = A_y, \ \ \ \ (A/\alpha)_f = A_f, \text{ for } f \in E'. \]
For a half edge $\gamma$ of $\Gamma$, we define the homomorphisms $(A/\alpha)_\gamma$ to coincide with $A_\gamma$ in case $\gamma$ is not attached to $x$ or $y$. If $\gamma$ was attached to $y$ in $\Gamma$, and lies on the edge $f$, we let $(A/\alpha)_\gamma$ be the map $A_\gamma : A_y \to A_f$ pre-composed with the identification of $A_y$ with $(A/\alpha)_{[xy]}$.
Finally, if $\gamma$ was attached to $x$ in $\Gamma$, we let $(A/\alpha)_\gamma$ be the composition:
\[\xymatrix{
(A/\alpha)_{[xy]} = A_y \ar[r]^-{A_\beta} & A_e \ar[r]^-{-A_\alpha^{-1}} & A_x \ar[r]^-{A_\gamma} & A_f = (A/\alpha)_f
}\]
The pair $(\Gamma/e, A/\alpha)$ is called the {\em contraction of the decorated graph} $(\Gamma, A)$ along $\alpha$. We say that $(\Gamma,A)$ {\em can be contracted to a point} if it can be reduced to a single vertex with no edges by performing a series of contractions of decorated graphs.

\begin{lemma}\label{contraction}
Let $(\Gamma, A)$ and the contraction $(\Gamma/e, A/\alpha)$ be as above. Then we have a natural morphism of complexes $\pi: \mathcal C(\Gamma/e, A/\alpha) \to \mathcal C(\Gamma, A)$ inducing an isomorphism on cohomology groups.
\end{lemma}
\begin{proof}
We define the morphism $\pi$ as follows. Let $\alpha, \beta$ be the half edges at $e$, attached to vertices $x, y$ of $\Gamma$. For a vertex $v \not\in \{x, y\}$ of $\Gamma$ and an edge $f \ne e$ of $\Gamma$, we define $\pi$ to map the groups $(A/\alpha)_v, (A/\alpha)_f$ to $A_v, A_f$ via the canonical identifications,
and to take $(A/\alpha)_{[xy]} = A_y$ to $A_x \oplus A_y \subset \mathcal C^0(\Gamma, A)$, by $a_y \mapsto (- A_\alpha^{-1} A_\beta(a_y), a_y)$. It is straightforward to check that this commutes with the differential -- i.e. is a morphism of complexes. One can also check that the cokernel of the map $C^0(\Gamma/e, A/\alpha) \to C^0(\Gamma, A)$ is exactly $A_x$, where the isomorphism is induced by the map $C^0(\Gamma, A)\rightarrow A_x$ given by $(a_v)_{v\in V} \mapsto A_\alpha^{-1} A_\beta(a_y) + x$.

The cokernel of the map $C^1(\Gamma/e, A/\alpha) \to C^1(\Gamma, A)$ is just $A_e$ via the projection map. We then obtain a map of short exact sequences
\[\xymatrix{
0 \ar[r] & \mathcal C^0(\Gamma/e, A/\alpha) \ar[r] \ar[d] & \mathcal C^0(\Gamma, A) \ar[r] \ar[d] & A_x \ar[d] \ar[r] & 0 \\
0 \ar[r] & \mathcal C^1(\Gamma/e, A/\alpha) \ar[r] & \mathcal C^1(\Gamma, A) \ar[r] & A_e \ar[r] & 0
}\]
Since the vertical map on the right is an isomorphism (it is the isomorphism $A_\alpha$), the snake lemma gives an isomorphism of the kernels and cokernels of the other two vertical maps. But these are the groups $H^0$ and $H^1$ of the complexes respectively.
\end{proof}

The following example illustrates the effect of contracting edges, and relates the cohomology of decorated graphs to that of the underlying graph as a topological space.

\begin{example}\label{graph coho ex}
Let $\Gamma$ be a graph, let $A_0$ be an abelian group, and let $A$ be a coefficient system on $\Gamma$ in which each $A_e$ and each $A_x$ is isomorphic to $A_0$ and in which every $A_\alpha$ is an isomorphism.  We may  wish to compare $H^1(\Gamma,A)$ (which depends on the choice of the isomorphisms $A_\alpha$) with
the group $H^1(\Gamma_{\operatorname{top}},A_0)$, where the latter group is the usual cohomology of the topological space $\Gamma_{\operatorname{top}}$ associated to the graph $\Gamma$, viewed as a finite simplicial complex, with coefficients in the group $A_0$.
Depending on the choice of the isomorphisms $A_\alpha$, these cohomology groups sometimes agree, and sometimes not, as illustrated below.\renewcommand{\theenumi}{\alph{enumi}}
\renewcommand{\labelenumi}{(\alph{enumi})}
\begin{enumerate}

\item  \label{simplicial ex}
We will first observe that the topological cohomology agrees with the cohomology of decorated graphs with respect to a particular choice of isomorphisms. To see this, for each edge $e$ of $\Gamma$ choose an orientation. Now if $e$ is an edge from vertex $x$ to vertex $y$, and if $\alpha$ is the half edge on $e$ attached to $y$, define $A_\alpha:A_y \to A_e$ to be the identity map on the group $A_0$; define $A_\beta:A_x \to A_e$ to be $-\mathrm{id}_{A_0}$ for the other half edge $\beta$ (attached to $x$).  Let $A_\Gamma$ be the coefficient system defined by these choices.  Then $H^1(\Gamma,A_\Gamma)=H^1(\Gamma_{\operatorname{top}},A_0)$, by the definition of simplicial cohomology.
\item \label{bipartite ex}
Consider a bipartite graph; i.e., a graph $\Gamma$ with vertex set
partitioned as $V = V' \sqcup V''$, such that each edge has a vertex in
$V'$ and a vertex in $V''$.  In this case, the cohomology of a constant
coefficient system with identity isomorphisms $A_\alpha$ agrees with
that of the topological space.  More precisely, take the coefficient
system $A$, in which each edge and vertex is decorated by the group $A_0$, and in which each homomorphism associated to a half edge is the identity. We may also consider the orientation in which each edge is chosen to start at a vertex in $V'$ and end in a vertex in $V''$. Let $A'$ be the coefficient system associated with this set of choices as in part~(\ref{simplicial ex}). Define a morphism of coefficient systems $\phi: A \to A'$ by choosing $\phi_x: A_x \to A'_x$ to be the identity if $x \in V''$ and the negative of the identity if $x \in V'$. Since $\phi$ is an isomorphism, part (\ref{simplicial ex}) yields $H^1(\Gamma, A) \simeq H^1(\Gamma, A') \simeq H^1(\Gamma_{\operatorname{top}},A_0)$.
\item \label{coho different ex}
As an example where $H^1(\Gamma,A_\Gamma)$ and $H^1(\Gamma_{\operatorname{top}},A_0)$ do not agree, let $\Gamma$ be a triangle; i.e., a graph consisting of three vertices $v_1,v_2,v_3$ and three edges $e_1,e_2,e_3$ with $v_i,v_j$ being the endpoints of $e_k$ for each permutation $(i,j,k)$ of $(1,2,3)$.  Let $A_0 = \Q$ and let each $A_\alpha$ be the identity.  Then $H^1(\Gamma_{\operatorname{top}},A_0) = \Q$, but
$H^1(\Gamma,A_\Gamma)=0$ by direct computation.  Note that $\Gamma$ is not bipartite, so this does not contradict part~(\ref{bipartite ex}).
\end{enumerate}
\end{example}

We give a sufficient condition for when a decorated graph can be contracted to a point:

\begin{prop}\label{contractible}
Let $(\Gamma, A)$ be a decorated graph. Assume that $\Gamma$  is a tree for which there is a choice of some vertex $v_0$ of $\Gamma$ (the ``root'') with the following property:
\begin{itemize}
\item[]
Suppose that $v,w$ are vertices of $\Gamma$ and $v$ is the parent of $w$ via an edge $e$ (i.e., $v,w$ are adjacent via $e$, and $v$ lies on the unique path connecting $v_0$ to $w$). Let $\alpha$ be the half edge associated to $w$ and $e$ (there is only one since a tree has no loops). Then $A_\alpha:A_w\rightarrow A_e$ is an isomorphism (equivalently, the half edge $\alpha$ is redundant).
\end{itemize}
Then $(\Gamma,A)$ can be contracted to a point.

\end{prop}
\begin{proof}
We show by induction on the number of vertices that $(\Gamma,A)$ can be contracted to a point. Let $w$ be a leaf of $(\Gamma, A)$ (i.e., a vertex of degree~$1$), and let $v$ be its parent in $\Gamma$ via an edge $e$. Let $\alpha$ be the half edge associated to $w$. By assumption, $\alpha $ is redundant. The contraction $(\Gamma/e, A/\alpha)$ is a decorated graph with fewer vertices; these vertices are the vertices of $\Gamma$ except with $w$ removed; the decorations of those remaining vertices are unchanged. Hence $(\Gamma/e, A/\alpha)$ also satisfies the hypothesis of the proposition, and so can be contracted to a point, by induction. Hence $(\Gamma, A)$ can be contracted to a point.\end{proof}

%================================================================
\section{Decorations of the Reduction Graph and Computations of $\Sha(F,T)$}\label{computations}
%================================================================

As before, consider a complete discrete valuation ring $R$ with fraction field $K$ and residue field $k$, and a semi-global field $F$ over $K$.  Let $\XX$ be a normal crossings model of~$F$ and let $\PP$ be a finite non-empty set of closed points of $\XX$ as in Section~\ref{setup}. Associated to $X^{\rm red}$ and $\PP$ is the so-called {\em reduction graph} $\Gamma$, as in \cite[Section~6]{HHK3} (see below).  In this section, we apply the above graph theoretic setup from the previous section to this graph.

%-----------------------------------------------------------------------------------------------
\subsection{Comparison of coefficient systems}\label{coeffs}
%-----------------------------------------------------------------------------------------------

Recall from Section~\ref{setup} that $\XX$ is a connected regular integral proper curve over $R$ with function field $F$ such that the reduced closed fiber $X^{\rm red}$ is a union of connected regular $k$-curves $C_i$ that intersect each other transversally, and that $\PP$ is a non-empty set of closed points of $\XX$ that contains the
(possibly empty) finite set ${\PP}_0$ of closed points at which two
components $C_i,C_j$ meet, and let $\UU$ be the set of connected
components of the complement of $\PP$ in the reduced closed fiber $X^{\operatorname{red}}$. Associated
to this is the reduction graph $\Gamma=\Gamma(\XX,\PP)$. This is a bipartite graph whose vertex set is $\PP\cup \UU$, and such that for each branch $(U,P)$ there is an edge connecting $U \in \UU$ to $P \in \PP$.

\begin{remark} \label{dual graph rk}
Recall that for a normal crossings model $\XX$, there is also another graph that can be associated to the reduced closed fiber, viz.\ the {\em dual graph} $\Gamma_{\operatorname{D}}$ (as in \cite[p.~86]{DM}).  Its vertices correspond to the irreducible components of the closed fiber, and its edges correspond to intersection points of those irreducible components.  If the closed fiber does not consist of just a single (regular) irreducible component, then the set of intersection points $\PP_0$ is non-empty; and the reduction graph $\Gamma$ associated to the reduced closed fiber and $\PP_0$ is the barycentric subdivision of $\Gamma_{\operatorname{D}}$ (see \cite[Remark~6.1(a)]{HHK3}). Even if $\PP$ is strictly larger than $\PP_0$, the graphs $\Gamma_{\operatorname{D}}$ and $\Gamma$ are homotopy equivalent as topological spaces.
\end{remark}

Associated to each $P$, $U$, and $(U,P)$ we have fields $F_P$,
$F_U$, and $F_{U,P}$ (see Section~\ref{setup}). For each $P$, we
also have the residue fields $\kappa(P)$ of $\XX$ at $P$. Recall that for  $U\in
\UU$, we defined $\kappa(U)=\OO (\bar U)$, the ring of functions on the
closure of $U$. Let $S$ be a flasque torus over $R$. If $P$ is in the closure of $U$, there is a specialization map $H^1(k(U),S)\rightarrow H^1(\kappa(P),S)$ (defined after Properties~\ref{properties_flasque}).
We define the following coefficient systems:

\newcommand{\fcoef}{\mc H_F(S)}
\newcommand{\kcoef}{\mc H_k(S)}
\newcommand{\kapcoef}{\mc H_{\kappa}(S)}
\begin{enumerate}
\item $\fcoef$: Each vertex $\xi$ in $\PP\cup \UU$ is decorated with
$H^1(F_\xi, S)$, and each edge $(U,P)$ is decorated with $H^1(F_{U,P}, S)$. The homomorphisms for the half edges are induced by the inclusions $F_P, F_U \hookrightarrow F_{U,P}$.
\item $\kcoef$:
Each vertex $P$ in $\PP$ is decorated with $H^1(\kappa(P), S)$, each
vertex $U$ in $\UU$ is decorated with $H^1(k(U), S)$, and each edge $(U,P)$ is decorated with
$H^1(\kappa(P),S)$. The homomorphisms for the half edges are the identity (for $P$) and the specialization map (for $U$), respectively.
\item $\kapcoef$:
Each vertex $\xi$ in $\PP \cup \UU$ is decorated with $H^1(\kappa(\xi), S)$,
and each edge $(U,P)$ is decorated with
$H^1(\kappa(P),S)$.  The homomorphisms for the half edges are induced by the identity (for $P$) and the natural inclusion map $\kappa(U)\rightarrow \kappa(P)$.
\end{enumerate}

\begin{lemma}\label{ShaFromGraph} Let $T$ be an $R$-torus
and let $1\rightarrow S\rightarrow Q\rightarrow T\rightarrow 1$ be a flasque resolution of~$T$.  Then $\Sha(F,T)\simeq H^1(\Gamma, \fcoef) \simeq H^1(\Gamma, \kcoef)$.
\end{lemma}

\begin{proof}
By Theorem~\ref{H1Scosets},
$$  \Sha_\PP(F, T)\simeq \left.  \prod _U H^1(F_U,S) \middle\backslash \prod_{(U, P)} H^1(F_{U, P},S)
 \middle\slash \prod_P H^1(F_P,S)  \right. .$$
The left hand side equals $\Sha(F,T)$ by Theorem~\ref{shas}.
The right hand side equals $ H^1(\Gamma, \fcoef)$ by definition.

For the second identification, we note that by
\ref{value}(\ref{value_a},\ref{value_b}), we may identify $H^1(F_P, S) \simeq H^1(\kappa(P), S)$ and
$H^1(F_{U, P}, S) \simeq H^1(\kappa(P), S)$. Since $H^1(\Gamma,
\fcoef)$ and $H^1(\Gamma, \kcoef)$ are therefore cokernels of homomorphisms to
the isomorphic groups $\prod_{(U, P)} H^1(F_{U, P}, S)$ and $\prod_{(U, P)}
H^1(\kappa(P), S)$, we therefore need only show that the images of
these homomorphisms coincide. But this follows from the fact that by
\ref{value}(\ref{value_b}), we have a
commutative diagram
\[\xymatrix{
H^1(F_P, S) \ar[r]^\sim \ar[d] & H^1(\kappa(P), S) \ar[d]^{=} \\
H^1(F_{U, P}, S) \ar[r]^\sim & H^1(\kappa(P), S),
}\] and from the fact that the images of
$H^1(F_U, S)$ and $H^1(k(U), S)$ in $\prod_{(U, P)} H^1(\kappa(P), S)$
coincide by Proposition~\ref{value}(\ref{value_c}).
\end{proof}

%----------------------------------------------------------------
\subsection{Applications to $\Sha(F,T)$}
%----------------------------------------------------------------

\renewcommand{\theenumi}{\alph{enumi}}
\renewcommand{\labelenumi}{(\alph{enumi})}

\begin{prop}\label{coefficient-relation}
Let $T$ be an $R$-torus
and let $1\rightarrow S\rightarrow Q\rightarrow T\rightarrow 1$ be a flasque resolution of $T$.
\begin{enumerate}
\item  \label{Sha upper bound}
There is a natural surjection
$$\phi: H^1(\Gamma, \kapcoef)\rightarrow H^1(\Gamma, \fcoef)=\Sha(F,T).$$
\item  \label{constant coeff sys}
If $A$ is a constant coefficient system, with value $H^1(\ell,S)$ on
each marking, where $\ell/k$ is a finite field extension, then
$$H^1(\Gamma, A)=H^1(\Gamma_{\operatorname{top}},
H^1(\ell,S))=\operatorname{Hom}(H_1(\Gamma_{\operatorname{top}},{\mathbb
Z}), H^1(\ell,S)),$$
where $\Gamma_{\operatorname{top}}$ denotes the topological space associated to $\Gamma$.
\item  \label{Sha for all lines}
If every component of the closed fiber $X$ of $\XX$ is isomorphic to a projective line (over its field of constants), then $\phi$ is an isomorphism.
\end{enumerate}
\end{prop}

\begin{proof}
By Lemma~\ref{ShaFromGraph}, $\Sha(F,T)\simeq H^1(\Gamma, \fcoef) \simeq
H^1(\Gamma, \kcoef)$. But there is a natural map of coefficient systems $\kapcoef
\to \kcoef$ which is the identity on the decorations for the edges $(U,
P)$. Since we have a commutative diagram
\[\xymatrix{
\mc C^0(\Gamma, \kapcoef) \ar[d] \ar[r] & \mc C^1(\Gamma, \kapcoef) \ar@{=}[d] \\
\mc C^0(\Gamma, \kcoef) \ar[r] & \mc C^1(\Gamma, \kcoef), \\
}\]
it follows that the image of the bottom map contains the image of the top map, which induces a surjective map of cokernels $H^1(\Gamma, \kapcoef)
\to H^1(\Gamma, \kcoef)$.
This proves part~(\ref{Sha upper bound}).

The second equality in part~(\ref{constant coeff sys}) is by the universal coefficient theorem, using that the coefficient system is constant. The first equality is the content of Example~\ref{graph coho ex}(\ref{bipartite ex}), because the reduction graph $\Gamma$ is bipartite.

Finally, for part~(\ref{Sha for all lines}), since
every component is isomorphic to a projective line,
we have an isomorphism $H^1(\kappa(U), S) \to H^1_{\text{\'et}}(U,S)$ by
Property~\ref{properties_flasque}(\ref{aff flasque coho iso}),
and a surjection $H^1_{\text{\'et}}(U,S) \to H^1(k(U),S)$ by
Property~\ref{properties_flasque}(\ref{flasque coho surj});
so the composition $H^1(\kappa(U), S) \to H^1(k(U), S)$ is surjective.
In particular, if we consider the coefficient system ${\mathcal K}$ defined by the
exact sequence $0 \to {\mathcal K} \to \kapcoef \to \kcoef \to
0$, then ${\mathcal K}$ is nonzero only on markings coming from vertices, not
edges. It follows that $H^1(\Gamma, {\mathcal K}) = 0$; and by the long exact sequence,
$H^1(\Gamma, \kapcoef) \simeq H^1(\Gamma, \kcoef) \simeq \Sha(F, T)$, with the
second isomorphism holding by Lemma~\ref{ShaFromGraph}.  This proves~(\ref{Sha for all lines}).
\end{proof}

Proposition~\ref{coefficient-relation} gives an upper bound on $\Sha(F,T)$, for $T$ an $\XX$-torus, in terms of the constant fields of the components of the reduced closed fiber $X^{\operatorname{red}}$ of our normal crossings model $\XX$ and the residue fields of the intersection points of the irreducible components of $X^{\operatorname{red}}$.  In particular, we immediately obtain the following:

\begin{theorem} \label{Sha via coef sys}
Let $\XX$ be a normal crossings model of a semi-global field $F$ over a complete discretely valued field $K$ with residue field $k$, and suppose that the reduced closed fiber $X^{\operatorname{red}}$ consists of $k$-curves $C_i$ that intersect only at $k$-points. Let $m$ be the number of cycles in the reduction graph (i.e., the rank of $H_1(\Gamma_{\operatorname{top}},{\mathbb Z})$). Let $T$ be an $R$-torus and let $1\rightarrow S\rightarrow Q\rightarrow T\rightarrow 1$ be a flasque resolution of~$T$.
\renewcommand{\theenumi}{\alph{enumi}}
\renewcommand{\labelenumi}{(\alph{enumi})}
\begin{enumerate}
\item The group $\Sha(F,T)$ is a quotient of $H^1(k,S)^m$.
\item\label{tree} If the reduction graph is a tree, then $\Sha(F,T)$ is trivial.
\item\label{all P^1} If each $C_i$ is isomorphic to $\P^1_k$, then $\Sha(F,T)$ is isomorphic to $H^1(k,S)^m$.  Hence
$\Sha(F,T)$ is trivial if and only if the reduction graph is a tree or $H^1(k,S)$ vanishes.
\end{enumerate}
\end{theorem}

The following theorem gives a generalization of the last part of Theorem~\ref{Sha via coef sys}. Its proof provides a general strategy for utilizing coefficient systems.

\begin{theorem}\label{shaP1}

Let $\XX$ be a normal crossings model of a semi-global field $F$ over a complete discretely valued field~$K$ with residue field~$k$, and suppose that the reduced closed fiber $X^{\operatorname{red}}$ consists of a union of copies of ${\mathbb P}^1_k$. Let $\PP$ be the set of intersection points (which might not be rational points). Let $\Gamma$ be the associated reduction graph. Then for any $R$-torus $T$ with flasque resolution $1\rightarrow S\rightarrow Q\rightarrow T\rightarrow 1$,  there is an exact sequence

$$\Hom(H_1(\Gamma_{\operatorname{top}},{\mathbb Z}), H^1(k,S))\rightarrow \Sha(F,T)\rightarrow \prod_{P\in \PP}\frac{H^1(\kappa(P),S)}{\mathrm{im}\, H^1(k,S)}\rightarrow 0.$$

\end{theorem}

\begin{proof}
If $\PP$ is empty, i.e., $X^{\operatorname{red}}={\mathbb P}^1_k$, the claim follows from Theorem~\ref{Sha via coef sys}(\ref{all P^1}). So we may assume that $\PP$ is not empty.
Consider the coefficient system $\kcoef$. As in the proof of Proposition~\ref{coefficient-relation}, one sees that there is a surjection $H^1(\kappa(U), S)=H^1(k,S) \to H^1(k(U), S)$ for each component~$U$. By specializing to a $k$-point one checks that this is also injective, so $H^1(k(U),S)\simeq H^1(k,S)$ for each~$U$.
Let $A_\bullet$ be the constant coefficient system given by decorating every vertex and edge with~$H^1(k,S)$, and let $C_\bullet$ denote the quotient of the coefficient system $\kcoef$ modulo $A_\bullet$. Since $H^1(k,S)=H^1(k(U),S)$ for each component $U$, $C_\bullet$ has trivial decorations at such~$U$. In particular, $H^0(\Gamma, C_\bullet)=0$.  Let $A'_\bullet$, $A''_\bullet$ be the kernel and image of $A_\bullet \to \mathcal{H}_k(S)$, respectively.  Applying Lemma~\ref{ses} to 
$0 \to A'_\bullet \to A_\bullet \to A''_\bullet \to 0$ and to $0 \to A''_\bullet \to \kcoef\to C_\bullet \to 0$ gives an exact sequence
$$H^1(\Gamma, A_\bullet)\rightarrow H^1(\Gamma,\kcoef)\rightarrow H^1(\Gamma,C_\bullet)\rightarrow 0.$$

The middle term is $\Sha(F,T)$ by Lemma~\ref{ShaFromGraph}. The left term is $\Hom(H_1(\Gamma_{\operatorname{top}},{\mathbb Z}), H^1(k,S))$ by Proposition~\ref{coefficient-relation}(\ref{constant coeff sys}).
To calculate the third term, recall that $C_\bullet$ has trivial decorations at components~$U$. The decoration at a point $P\in \PP$ is $\frac{H^1(\kappa(P),S)}{H^1(k,S)}$, and the decoration at a branch $(U,P)$ is also $\frac{H^1(\kappa(P),S)}{H^1(k,S)}$. The third term is the quotient of $\prod_{(P,U)} \frac{H^1(\kappa(P),S)}{H^1(k,S)}$ by $\prod_{P\in \PP}\frac{H^1(\kappa(P),S)}{H^1(k,S)}$. But each $P$ lies on exactly two branches, so that quotient is  $\prod_{P\in \PP}\frac{H^1(\kappa(P),S)}{H^1(k,S)}$ as claimed.
\end{proof}
This theorem will be used in Section~\ref{exex} to provide explicit examples where $\Sha(F,T)$ is non-trivial.

%==========================================================
\section{Monotonic reduction graphs and the vanishing of $\Sha(F,T)$}\label{monotonic}
%==========================================================

We next obtain a vanishing result for $\Sha(F,T)$ for a special type of reduction graph, which we call a ``monotonic tree''.  We also give two alternate descriptions of such trees, and we give examples to show that other natural tree-like assumptions on the reduction graph do not suffice to guarantee the vanishing of $\Sha(F,T)$.

%---------------------------------------------------------------------------------------------------
\subsection{Vanishing of $\Sha(F,T)$}\label{vanishing}
%---------------------------------------------------------------------------------------------------

Let $F$ be a semi-global field with normal crossings model $\XX$, and let $\Gamma$ be the associated reduction graph.
To each vertex $v$ of the reduction graph $\Gamma$ we associate the field $\kappa(v)$.  We say that the reduction graph $\Gamma$ is {\em monotonic} if it is a tree and there is a choice of some vertex $v_0$ (the ``root'') with the following property:
If $v,w$ are vertices of $\Gamma$ and $v$ is the parent of $w$ (i.e., $v,w$ are adjacent, and $v$ lies on the unique path connecting $v_0$ to $w$)
then $\kappa(v) \subseteq \kappa(w)$.
Recall that if $P \in \PP$ lies on the closure of $U \in \UU$, then $\kappa(P)$ necessarily contains $\kappa(U)=\OO(\bar U)$; hence if $P$ is the parent of $U$ on a monotonic tree
then necessarily $\kappa(P)=\kappa(U)$.

\begin{remark}
If $\XX$ is a normal crossings model with associated reduction graph $\Gamma$ as above, it is easy to check that the reduction graph of any blow-up of $\XX$ is monotonic if and only if $\Gamma$ is monotonic. For any two normal crossings models ${\XX}_i$, $i=1,2$,
of~$F$ there is a normal crossing model ${\XX}$ of $F$
and morphisms $\XX \to {\XX}_i$ which are
given by a series of blow-ups. Indeed, one first considers the
closure ${\mc Z} \subset {\XX}_1\times_R {\XX}_2$
of the diagonal of the generic fibers. By \cite[page 193]{Lip75}
there exist a normal crossings model ${\XX}$ and
a proper birational $R$-morphism ${\XX}\to {\mc Z}.$
By \cite{lichtenbaum}, Theorem II.1.15, each composite morphism
${\XX}\to {\mc Z} \to {\XX}_i$ is a composite of
blow-ups in closed points.
So for a semi-global field $F$, either all associated reduction graphs are monotonic or none of them is.
\end{remark}

Let $S$ be an $R$-torus, and let $\kapcoef$ be the coefficient system on $\Gamma$ defined in Subsection~\ref{coeffs}. From Proposition~\ref{contractible}, we immediately obtain:
\begin{lemma}\label{monocontractible}
With $F, \XX, \Gamma$ as above, assume that $\Gamma$ is a monotonic tree. Then $(\Gamma, \kapcoef)$ can be contracted to a point.
\end{lemma}
\begin{proof}
The decorated graph $(\Gamma, \kapcoef)$ satisfies the condition in Proposition~\ref{contractible}, since it is monotonic, the decoration of each edge $(U,P)$ is $H^1(\kappa(P),S)$ by definition, and  $\kappa(P)=\kappa(U)$ if $P$ is the parent of~$U$.
\end{proof}

\begin{theorem} \label{triv Sha monotonic}
Let $K$ be a complete discretely valued field with valuation ring~$R$, and let $F$ be a semi-global field over~$K$. Let $\XX$ be a normal crossings model of $F$, and assume that the associated reduction graph $\Gamma$ is a monotonic tree. Then for any  $R$-torus $T$,  $\Sha(F,T)$ is trivial.
\end{theorem}

\begin{proof}
Let $1\rightarrow S\rightarrow Q\rightarrow T\rightarrow 1$ be a flasque resolution of $T$.  We then have a coefficient system $\kapcoef$ on $\Gamma$ as above.
By Proposition~\ref{coefficient-relation}, there is a surjection $H^1(\Gamma, \kapcoef)\rightarrow \Sha(F,T)$. So it suffices to show that  $H^1(\Gamma, \kapcoef)$ is trivial. By Lemma~\ref{monocontractible}, $(\Gamma, \kapcoef)$ can be contracted to a point. But Lemma~\ref{contraction} asserts that contraction does not change the cohomology of the graph. That is, $H^1(\Gamma, \kapcoef)$ is the cohomology of a (decorated) graph consisting of a single vertex and no edge and thus it is trivial, as we wanted to show.
\end{proof}

\begin{remark}
Note that in the situation of Theorem~\ref{Sha via coef sys}(\ref{tree}), the reduction graph is a monotonic tree, with all labels equal to $k$. Hence this illustrates the theorem above.
\end{remark}
%----------------------------------------------------------------------
\subsection{Characterizations of monotonic trees}
%----------------------------------------------------------------------

In this subsection, we give two characterizations of monotonic trees.

In the above situation, let $\PP_0$ be the subset of $\PP$ consisting of the points of $X^{\rm red}$ where two irreducible components meet.  We obtain a subgraph $\Gamma_0$ of $\Gamma$ having vertices $\PP_0 \cup \UU$, and having an edge connecting any two vertices that are adjacent in $\Gamma$.
The graph $\Gamma_0$ is a tree if and only if $\Gamma$ is; and in that case, $\Gamma_0$
is monotonic if and only if $\Gamma$ is.

\begin{prop}
Let $\XX$ be a normal crossings model for a semi-global field $F$, and suppose that the associated reduction graph $\Gamma$ is a tree.
Then $\Gamma$ is monotonic if and only if there is an injection $\psi:\PP_0 \to \UU$ such that for every $P \in \PP_0$, $P \in \overline{\psi(P)}$ and $\kappa(P)=\kappa(\psi(P))$.
\end{prop}

\begin{proof}
For the forward direction, note that any $P \in \PP_0$ lies on exactly two irreducible components of $X$.   So if $\Gamma$ is monotonic then the vertex $P$ has exactly one parent $U \in \UU$ and one child $U' \in \UU$ (i.e., $P$ is the parent of $U'$), with $P$ in the closure of each.  As noted above, $\kappa(P)=\kappa(U')$.  Thus we may define $\psi(P) = U'$.  This map is injective because a vertex cannot have more than one parent.

For the reverse direction, consider the subtree $\Gamma_0$ of $\Gamma$ defined above.
This is a bipartite tree with vertex sets $\PP_0$ and $\UU$, and
its terminal vertices all lie in $\UU$.  By induction one sees that such a bipartite tree has the property that $|\UU| = |\PP_0| + 1$.  It follows from injectivity that there is a unique element $U_0 \in \UU$ that is not in the image of $\psi$; and we will show that $\Gamma$ is monotonic with respect to the root $U_0$.

Take any path in $\Gamma$ beginning at $U_0$, with consecutive vertices $v_0,v_1,\dots, v_n$, where $v_0=U_0$.  Thus $v_i \in \UU$ for $i$ even, and $v_i \in \PP$ for $i$ odd.  For $i$ odd, the two elements of $\UU$ whose closures contain $v_i \in \PP$ are $v_{i-1}$ and $v_{i+1}$.  But $v_0=U_0$ is not in the image of $\psi$.  So $\psi(v_1)=v_2$.  Since $\psi$ is injective, it follows by induction that $\psi(v_i) = v_{i+1}$ for all odd $i$.  Thus $\kappa(v_i) = \kappa(v_{i+1})$ for $i$ odd.  Also, $\kappa(v_i) \subseteq \kappa(v_{i+1})$ for $i$ even, since in that case $v_{i+1} \in \PP$ is in the closure of $v_i \in \UU$.  Thus the monotonicity condition holds for every pair of adjacent vertices in this path.  Since every pair of adjacent vertices in $\Gamma$ lies in some path that begins at $U_0$, the result follows.
\end{proof}

The following proposition relates the notion of a monotonic tree to the tree property of the reduction graph for the curve and its base change to finite field extensions of the residue field.
By Zariski's connectedness theorem, $X$ is geometrically connected;  i.e., $X_{k'}$ is connected for any finite field extension $k'/k$ (\cite[Corollaire~III.4.3.12]{EGA3}).  Given a finite nonempty subset $\PP$ of $X^{\operatorname{red}}$, let $\PP_{k'}$ denote the subset of $X_{k'}$ consisting of the points that lie over points of $\PP$. The graph $\Gamma'$ associated to $X_{k'}$ and $\PP_{k'}$ is defined in the natural way, as for $X$ and $\PP$. If $\Gamma'$ is a tree then so is the original reduction graph $\Gamma$, because the inverse image of any loop in $\Gamma$ would
contain a loop in $\Gamma'$.
The converse is more subtle, as discussed in the following proposition (see also Example~\ref{geom not monotonic ex} below).

\begin{prop} \label{monotonic geometric}
Let $R$ be a complete discrete valuation ring with fraction field $K$ and residue field $k$, and let $\XX$ be a normal crossings model for a semi-global field $F$ over $K$.
Let $X$ be the closed fiber of $\XX$.  Let $\Gamma$ be the reduction graph associated to $X^{\operatorname{red}}$ and a finite set $\PP \subset X^{\operatorname{red}}$ as above.
\renewcommand{\theenumi}{\alph{enumi}}
\renewcommand{\labelenumi}{(\alph{enumi})}
\begin{enumerate}
\item \label{mono implies geom}
If $\Gamma$ is a monotonic tree then for every finite field extension $k'/k$, the graph associated to $X_{k'}$ and $\PP_{k'}$ is a tree.
\item \label{geom implies mono}
In part~(\ref{mono implies geom}), suppose that each of the fields $\kappa(P)$ is a separable extension of $k$ (e.g.\ if $k$ is of characteristic zero, or more generally perfect).  Then the converse also holds: if the graph associated to $X_{k'}$  and $\PP_{k'}$ is a tree for every finite field extension $k'/k$, then $\Gamma$ is a monotonic tree.
\end{enumerate}
\end{prop}

\begin{proof}
For part (\ref{mono implies geom}), let $k'/k$ be a finite field extension, and consider the graph $\Gamma'$ associated to the base change $X_{k'}$ of $X$ and $\PP_{k'}$ as introduced above. Let $\pi:\Gamma' \to \Gamma$ be the natural map. Since $X_{k'}$ is connected, so is the graph $\Gamma'$.

As before, let $\UU$ be the set of connected components of the complement of $\PP$ in $X^{\operatorname{red}}$.  If
$U \in \UU$, then $\bar U_{k'} = \pi^{-1}(\bar U)$ consists of a disjoint union of copies of an irreducible $k'$-curve.  Thus a point of $X_{k'}$ can lie on at most one irreducible component of $\bar U_{k'}$.  Moreover, if $P \in \PP$ lies in the closure of $U \in \UU$, with $\kappa(P) = \kappa(U)$, and if $U'$ is a component of $\bar U_{k'}$, then there is exactly one point of $\PP_{k'}$ over $P$ that lies on $U'$.  Together, these observations show that if $v,w$ are adjacent vertices of $\Gamma$, with $v$ the parent of $w$, and if $w'$ is a vertex of $\Gamma'$ lying over $w$, then there is a unique vertex $v'$ of $\Gamma'$ that is adjacent to $w'$ and that lies over $v$.

Pick a vertex $v_0'$ of $\Gamma'$ over the root $v_0$ of $\Gamma$; i.e., $\pi(v_0')=v_0$.
Now consider any path $v_0',v_1',\dots,v_n'$ in $\Gamma'$ starting at $v_0'$.  (As before, {\em paths} are assumed to have no repeated edges.)  Let $v_i = \pi(v_i')$.  Then $v_i$ is the parent of $v_{i+1}$ for all $i < n$, since otherwise $v_{i+1}$ would be the parent of $v_i$ which inductively would give $v_{i+1}=v_{i-1}$, contradicting the uniqueness assertion in the previous paragraph.  Thus for every path in $\Gamma'$ from $v_0'$ to a vertex $v'$, its image in $\Gamma$ is a path such that each vertex is the parent of its successor (this being the unique path connecting $v_0$ to $\pi(v')$).  By the uniqueness assertion of the previous paragraph, for each vertex $v'$ of $\Gamma'$, there is a unique path from $v_0'$ to $v'$ in $\Gamma'$, and it lies over the unique path from $v_0$ to $\pi(v')$ in $\Gamma$.

Now consider any edge $e'$ in $\Gamma'$, with vertices $v',w'$.  Let $e,v,w$ be their images in~$\Gamma$.  We may assume that $v$ is the parent of $w$.  So the unique path from $v_0$ to $w$ ends with the edge $e$ from $v$ to $w$; and concatenating $e'$ with the unique path from $v_0'$ to $v'$ yields the unique path from $v_0'$ to $w$.  Hence if $e'$ is deleted from $\Gamma'$, then there is no path connecting $v_0$ to $w'$.  Thus if any edge is deleted from $\Gamma'$, the resulting graph is disconnected.  This implies that the connected graph $\Gamma'$ is a tree.

For part (\ref{geom implies mono}), let $k'$ be a finite Galois extension of $k$ that contains all the fields $\kappa(P)$ for $P \in \PP$.  Thus $k'$ also contains the fields $\kappa(U)$ for $U \in \UU$, since $\kappa(U) \subseteq \kappa(P)$ if $P \in \bar U$.  The graph $\Gamma'$ associated to $X_{k'}$ is a tree, acted upon by $\Gal(k'/k)$, and with the property that $\kappa(v') = k'$ for every vertex $v'$ of $\Gamma'$.  Since $\Gamma'$ is bipartite, no two adjacent vertices can be interchanged by an element of $\Gal(k'/k)$.  By
\cite[Theorem~I.6.1.15]{Serre}, it follows that there is a vertex that is fixed by $\Gal(k'/k)$.  The image of this vertex under $\pi:X_{k'} \to X$ is a vertex $v_0 \in \PP \cup \UU$ of $\Gamma$ with $\kappa(v_0)=k$.
To prove part~(\ref{geom implies mono}), we will show that $\Gamma$ is monotonic with root $v_0$.  For this, it suffices to show that if $v_0,v_1,\dots,v_n$ are consecutive vertices of a path in $\Gamma$ (with $v_0$ as above), then $\kappa(v_{n-1}) \subseteq \kappa(v_n)$.  If this does not hold, then $\kappa(v_n)$ is strictly contained in $\kappa(v_{n-1})$, with $v_{n-1} \in \PP$ and $v_n \in \UU$, since $\kappa(U) \subseteq \kappa(P)$ for $P \in \PP$ in the closure of $U \in \UU$.  Pick a path $v_0',v_1',\dots,v_n'$ in $\Gamma'$ with $v_i'$ lying over $v_i$; this exists because if $P \in \PP$ lies in the closure of $U \in \UU$, then every point of $\PP_{k'}$ over $P$ lies in the closure of some element of $\UU_{k'}$ over $U$ (i.e., an irreducible component of
$\pi^{-1}(U)$) and vice versa.  Let $\sigma$ be a non-trivial element of $\Gal(k'/\kappa(v_n))$ that does not lie in the proper subgroup $\Gal(k'/\kappa(v_{n-1}))$.  Then $\sigma$ fixes $v_0$ and $v_n$ but does not fix $v_{n-1}$; and so it carries the above path to a different path in $\Gamma'$ with the same endpoints.  The concatenation of one of these paths with the inverse of the other contains a loop in $\Gamma'$, and this is a contradiction.
\end{proof}

The following example shows that if $k'/k$ is a finite field extension, then the graph
$\Gamma'$ associated to the base change $X_{k'}$ need not be a tree even if $\Gamma$ is, when $\Gamma$ is not monotonic.
The example also shows that the hypothesis in Proposition~\ref{monotonic geometric}(\ref{geom implies mono}) is necessary.  That is, it is
possible in non-zero characteristic for a tree not to be monotonic, even if the graph associated to every base change of the closed fiber is a tree.

\begin{example}\label{geom not monotonic ex}
Let $k$ be a field, and let $R = k[[t]]$, with fraction field $K=k((t))$.  Let $a$ be a non-square in $k$
and let $\XX = \operatorname{Proj}(R[x,y,z]/((y-x)(xy - az^2) + tz^3))$, with function field $F$.  Then $\XX$ is a normal crossings model of $F$, whose closed fiber $X$ consists of two irreducible components $C_1,C_2$, each isomorphic to $\P^1_k$, with $\kappa(C_i) = k$, and intersecting at a single point $P$ with residue field $k' := \kappa(P) = k(\sqrt a)$.  Thus the associated reduction graph $\Gamma$ is a tree, but it is not monotonic.  If $\operatorname{char}(k) \ne 2$, then the graph $\Gamma'$ associated to the base change $X_{k'}$ is not a tree,
since $P$ splits into two points over $k'$.  But if $\operatorname{char}(k) = 2$, then every base change again gives a tree since $P$ remains a single point, despite the graph not being monotonic.
\end{example}

%-------------------------------------------------------------------------------------------------------------------------------------------------------------------------------------------------------------------------------------------------------------------------
\section{Examples}\label{exex}
%-------------------------------------------------------------------------------------------------------------------------------------------------------------------------------------------------------------------------------------------------------------------------

As before, $K$ is a complete discretely valued field with valuation ring $R$.

In this section, we produce examples of semi-global fields $F$ over $K$ and tori $T$ for which $\Sha(F,T)$ is nontrivial or even infinite (Example~\ref{infinite}).
By \cite[Theorem~4.2]{HHK3}, $\Sha_{\PP}(F,T)$ is trivial when $T$ is a rational $F$-torus. At the end of the section, we give an example
of an $R$-torus for which $T\times_{R}K$ is not $K$-rational,
the reduction graph has loops, and $\Sha(F, T)=0$ (Example~\ref{loop trivial Sha}).

A major ingredient in constructing examples where $\Sha(F,T)\neq 0$ will
be Proposition~\ref{coefficient-relation} and its consequences,
Theorems~\ref{Sha via coef sys} and~\ref{shaP1}.  In order to apply
those results, we will want to find examples of fields $k$ and flasque
$k$-tori $S$ for which $H^1(k,S)$ is nontrivial, or such that
$H^1(k',S)$ is strictly larger than $H^1(k,S)$ for some finite extension
$k'/k$.

%-------------------------------------------------------------------------------------------------------------------------------------------------------------------------------------------------------------------------------------------------------------------------
\subsection{Nontrivial $H^1(k,S)$}
%-------------------------------------------------------------------------------------------------------------------------------------------------------------------------------------------------------------------------------------------------------------------------

Let $k$ be a field and let $L/k$ be a finite Galois extension with
Galois group $G$.
Let $T=R^1_{L/k}\G_{m}$ be the $k$-torus defined by the equation ${\rm
Norm}_{L/k}(\xi)=1$.  In the proof of \cite[Prop. 15, p. 206]{CTS77}, an explicit flasque resolution $1 \to S \to Q \to T \to 1$ is constructed. It induces an isomorphism between $H^1(k,S)=T(k)/{\rm R}$ and
$\wh{H}^{-1}(G,L^\times)$. The latter is the quotient
${}^{N}L^\times/I_{G}L^\times$ of the group ${}^{N}L^\times$ of norm 1
elements in $L^\times$
by the subgroup $I_{G}L^\times$ of $L^\times$ generated by elements of the form
$\sigma(x)/x$, for $\sigma\in G$ and $x\in L^\times$ (see
\cite[Section~I.2]{NSW}).
If $L/k$ is a biquadratic Galois extension, one checks that $H^1(k,S)=
\wh{H}^{-1}(G,L^\times)$ is annihilated by 2.

 \begin{example}\label{anconic}
 Let $k$ be a field of characteristic not equal to 2 and $L = k(\sqrt{a}, \sqrt{b})$, for $a,b\in k$.
 Suppose that $[L : k] = 4$. Let $\sigma, \tau$ be generators of the Galois group so that $\sigma$ fixes $\sqrt{b}$ and $\tau$ fixes $\sqrt{a}$. Suppose moreover that the quadratic form $\langle 1, a, -b\rangle=x^2+ay^2-bz^2$ is anisotropic over $k$.
 Let $T= R^1_{L/k}({\mathbb G}_m)$. We claim that if $\sqrt{-1} \in k$, the class of $\sqrt{-1}$ in $ {}^{N}L^\times/I_{G}L^\times$
 is nontrivial. Indeed, first notice that $\operatorname{N}_{L/k}(\sqrt{-1})=1$, so $\sqrt{-1}\in  {}^{N}L^\times$. Now suppose for contradiction that $\sqrt{-1}\in I_{G}L^\times$. Since $\frac{\tau\sigma(z)}{z}=\frac{\sigma (z)}{z}\frac{\tau(\sigma (z))}{\sigma(z)}$ for all $z\in L^\times$, $I_{G}L^\times=\{\frac{\sigma(x)}{x}\frac{\tau(y)}{y}|\; x,y \in L^\times\}$. So by assumption, $\sqrt{-1}=\frac{\sigma(x)}{x}\frac{\tau(y)}{y}$ for some $x,y \in L^\times$. The element $\theta:=\frac{\sigma(x)}{x}$ satisfies $\operatorname{N}_{L/k(\sqrt{b})}(\theta)=1$; and $\operatorname{N}_{L/k(\sqrt{a})}(\theta)=\operatorname{N}_{L/k(\sqrt{a})}(\sqrt{-1}\frac{y}{\tau(y)})=-1$. After writing $\theta$ as $\theta =z_1+z_2\sqrt{a}+z_3\sqrt{b}+z_4\sqrt{ab}$ with coefficients $z_i\in k$, these equalities on the norms give a system of equations which implies $z_1^2-abz_4^2=0$, and hence $z_1=z_4=0$ (since $ab$ is not a square in $k$). Again using the system of equations, this in turn implies $bz_3^2-az_2^2=1$.  But the latter is impossible since  $\langle 1, a, -b\rangle$ is anisotropic over $k$.
 \end{example}

  \begin{example}\label{csv}
  Let $L=\Q(\sqrt{-1},\sqrt{2})$.
Let $T=R^{1}_{L/\Q}({\mathbb G}_{m})$, and let $1\rightarrow S\rightarrow Q\rightarrow T\rightarrow 1$ be a flasque resolution.
Let  $p_{i}, i \in \N$, be the infinite list of primes  congruent to $1$ modulo 8.
For any $n$, let $k_{n}: =\Q(\sqrt{p_{1}}, \dots, \sqrt{p_{n}})$. We have obvious
embeddings $k_{n}  \subset k_{n+1}$.
Each $p_{i}$ is a square in the 2-adic field $\Q_{2}$, hence
there are compatible embeddings $k_{n} \subset \Q_{2}$.
 Let $k_\infty=\cup_{n=1}^{\infty} k_{n}$. Using class field theory, one shows that $H^1(\Q, S) = 0$,
that each of the groups $H^1(k_{n},S)=T(k_n)/{\rm R}$
is finite of order $2^{2^n-1}$  (\cite[Cor. 2 p. 207]{CTS77}), 
 and that the natural maps
$H^1(k_{n},S) \to H^1(k_{n+1},S)$ are injective (\cite[Appendix, \S 2, Teorema 18, p. 206]{Vosk}).
Hence $H^1(k_\infty,S)$ is infinite.
  \end{example}

We also give an example of an infinite first cohomology group of a flasque torus that does not rely on algebraic number theory.

 \begin{prop}\label{genrericinfinite}
Let $1 \to S \to Q \to T \to 1$ be a flasque resolution of a $k$-torus $T$.
 Let $k(T)$ be the function field of $T$.
If $H^1(k,S) \neq 0$, then the natural  map $H^1(k,S) \to H^1(k(T),S)$
is injective but not surjective.
\end{prop}
\begin{proof}
For any integral $k$-variety with a smooth $k$-point,  with function field $k(X)$,
and any $k$-torus $R$,
the natural map $H^1(k,R) \to H^1(k(X),R)$ is injective, by an easy specialization argument.
The flasque resolution of $T$
 defines a torsor over $T$ under $S$, hence defines a class $\xi \in H^1(T,S)$.
We have the map $\delta : T(k) \to H^1(k,S)$ attached to the above exact sequence.
For any point $P \in T(k)$, we have the equality $\delta (P)= \xi(P) \in H^1(k,S)$.
One may also compute $\delta (P)$ by restricting the class $\xi$ to the local ring at $P$,
then using the map $H^1(\OO_{T,P},S) \to H^1(k,S)$.
Let $P \in T(k)$. We have maps
$$ H^1(k,S) \to H^1(T,S) \to H^1(\OO_{T,P},S) \to H^1(k(T),S).$$
Because $S$ is flasque, the map $H^1(\OO_{T,P},S) \to H^1(k(T),S)$
is an isomorphism.
If  the map $H^1(k,S) \to H^1(k(T),S)$ is onto, then
all maps in
$H^1(k,S) \to H^1(\OO_{T,P},S) \to H^1(k(T),S)$
are isomorphisms. Thus there exists a fixed $\xi_{0} \in H^1(k,S)$
whose image is $\xi \in H^1(k(T),S)$ and we get
$\xi(P)=\xi_{0} \in H^1(k,S)$, hence $\delta (P)=\xi_{0}$.
If $H^1(k,S) \neq 0$, since the map $T(k) \to H^1(k,S)$ is onto,
there exist $k$-points $P$ and $Q$ such that $\delta(P) \neq \delta(Q) \in H^1(k,S)$.
Thus  $H^1(k,S)\neq 0$ implies that the natural map $H^1(k,S) \to H^1(k(T),S)$
is not onto.
\end{proof}

\begin{example}\label{algebraic-infinite}
Let $k$ be a field and let $S$ be a flasque torus over $k$ with $H^1(k,S) \neq 0$.
One may then produce an exact sequence of tori $1 \to S \to Q \to T \to 1$ with $Q$ quasitrivial, by letting $\hat T$ be the kernel of an equivariant surjection from a permutation module to the character group $\hat S$.  Using Proposition \ref{genrericinfinite},  passing from $k$ to $k(T)$ and iterating the process (with the same exact sequence), we produce a field $E$ of infinite transcendence degree over $k$ such that $H^1(E,S)$ is infinite. \end{example}

\begin{example}\label{charnot2}
 Let $k = {\mathbb R}((x))((y))$ and $L = k(\sqrt{x}, \sqrt{y})$.
  Let $T = R^1_{L/k}({\mathbb G}_m)$, and let $1 \to S \to Q \to T \to 1$ be a flasque resolution of~$T$. Let $k' = {\mathbb C}((x))((y))$.
Then  one can show that $H^1(k, S) = 0$ and $H^1(k', S) \neq 0$. In particular the map $H^1(k,S)  \to  H^1(k',S)$   is not onto.
\end{example}

\begin{prop}
 \label{char2}
Let $k$ be a local  field of characteristic equal to 2.  Then $k = \F((s))$ for some finite field $\F$ with char$(\F) = 2$.
Let $k' = k(\sqrt{s})$.
Let $L/k$ be the biquadratic Galois extension of $k$ with Galois group $G=\Z/2 \times \Z/2$, and let $T = R^1_{L/k}{\mathbb G}_m$.
 Let $1 \to S \to Q \to T \to 1$ be a flasque resolution of~$T$.
 Then the map $H^1(k,S) \to H^1(k',S)$ is not onto.
\end{prop}
\begin{proof}   Recall   $H^1(k,S) \simeq {}^{N}L^*/I_{G}L^*$.
Let $\alpha \in L^*$ with $N_{L/k}(\alpha) = 1$. Since $k$ is a local field of
characteristic 2, $L^* \subset L(\sqrt{s})^{*2}$.  Hence $\alpha = \beta^2$ for some $\beta \in L(\sqrt{s}) = Lk'$.
Since $N_{L/k}(\alpha) = 1$, $N_{Lk'/k'}(\beta)^2 = 1$. Since char$(k) = 2$, we have $N_{Lk'/k'}(\beta) = 1$.
Thus $\beta$ gives rise to an element in $H^1(k',S)$.
Since  $H^1(k',S)$ is 2-torsion,  $\alpha = \beta^2$ is trivial in $H^1(k',S)$. Thus the natural map
$H^1(k,S) \to H^1(k',S)$ is the trivial map.
But $H^1(k',S) = \wh{H}^{-1}(G,(Lk')^*)$; and by \cite[Theorem
7.2.1]{NSW}, the latter group is dual to $\wh{H}^3(G,\Z) =
H^3(G,\Z)\simeq \Z/2$.
Thus the map $H^1(k,S) \to H^1(k',S)$  is
not onto.
\end{proof}

%-----------------------------------------------------------------------------------------------
\subsection{Nontrivial $\Sha(F,T)$}\label{nontriv sha ex}
%-----------------------------------------------------------------------------------------------

Let $T$ be a smooth $R$-torus and let $F$ be a semi-global field over
the fraction field of $R$.  In Theorem~\ref{triv Sha monotonic}, we
proved that $\Sha(F,T)$ is trivial if the reduction graph of a normal
crossings model of $F$ is a monotonic tree.  In Examples~\ref{infinite}
and~\ref{sha varies base change}, we show that if the reduction graph is
not a tree, then $\Sha(F,T)$ can  behave quite differently.  Afterwards,
in Example~\ref{non-monotonic Sha}, we show that a non-monotonic tree
can also lead to a non-trivial $\Sha(F,T)$.

\begin{example}\label{infinite}
Let $k$ be a field, let $R= k[[t]]$, with fraction field $K= k((t))$.
Let $\XX = \operatorname{Proj}(R[x,y,z]/(xyz-t(x+y+z)^3))$ and $F$ be the function field of $\XX$.
Then $\XX$ is a regular normal crossings model of $F$, whose (reduced) closed fiber
 $X$ consists of three irreducible components $C_1,C_2,C_{3}$, each isomorphic to $\P^1_k$, with $\kappa(C_i) = k$,
any two of which meet transversely at a single $k$-point. The reduction graph $\Gamma$  is a triangle, hence not a tree.
It satisfies $H_{1}(\Gamma_{\operatorname{top}},\Z)=\Z$.

Let $T$ be a $k$-torus, and let $1 \to S \to Q \to T \to 1$ be a flasque resolution.
 Then by Theorem~\ref{Sha via coef sys}(\ref{all P^1}),
we have $\Sha(F,T)  \simeq H^1(k,S)$.
 If $k$ and $T$ are as in Example \ref{anconic} or as in Example \ref{csv},
 then $\Sha(F,T) \neq 0$.
If $k$ and $T$ are as in Example~\ref{csv},
 then $\Sha(F,T)$ is infinite. Similarly, Example~\ref{algebraic-infinite} can be used to produce an example of infinite $\Sha(F,T)$ (with residue field $k=E$).
\end{example}

We now give an example of a semi-global field $F$ and an $R$-torus with
$\Sha(F, T) $  trivial and $\Sha(L, T)$  non-trivial for some finite field extension $L/F$ (which in fact comes from an extension of $k$).

\begin{example}\label{sha varies base change}
Let $k$ be a field, $R=k[[t]]$, and let $\XX$ and $F$ be as in Example~\ref{infinite}.
 Let $T$ be a $k$-torus, and let $1 \to S \to Q \to T \to 1$ be a flasque resolution.
 Then by Theorem~\ref{Sha via coef sys}(\ref{all P^1}), we have $\Sha(F,T)  \simeq H^1(k,S)$.  Now let $k$, $T$  and $k'$ be  as in Example \ref{charnot2}; or  let $k = \Q$, let $T$ be  as in Example~\ref{csv} and let $k' = \Q(\sqrt{17})$.
 Then  $\Sha(F,T)  \simeq H^1(k, S) =  0$ and $\Sha(F\otimes_k k',T) \simeq H^1(k', S)  \neq 0$.
\end{example}

We next give two related examples to show that $\Sha(F,T)$ can be
nontrivial even when the reduction graph is a tree, if it is not a
monotonic tree.  As we saw in Proposition~\ref{monotonic geometric}, a
reduction graph that is a tree can fail to be monotonic if the reduction
graph of some base change is not a tree, or if the residue field of $R$
has finite characteristic.  Our two examples illustrate each of those
two situations.

\begin{example} \label{non-monotonic Sha}
Let $T$ be a torus over a field $k$, let $R=k[[t]]$, and let $\XX$ and
$F$ be as in Example~\ref{geom not monotonic ex}.  We consider two cases:
\renewcommand{\theenumi}{\alph{enumi}}
\renewcommand{\labelenumi}{(\alph{enumi})}
\begin{enumerate}
\item \label{ch 0 nontriv Sha}
Let $k$, $T$, and $k'$ be as in Example~\ref{charnot2}, so that $k$ is
of characteristic zero.  As noted in Example~\ref{geom not monotonic
ex}, the reduction graph $\Gamma$ of $\XX$ is a tree, but not a
monotonic tree, since the reduction graph of the base change to $k'$ is
not a tree.  Since $H_{1}(\Gamma_{\operatorname{top}}, \Z)=0$,
Theorem \ref{shaP1} yields that $\Sha(F,T) \simeq H^1(k',S)/H^1(k,S) \ne 0$.
\item \label{ch 2 nontriv Sha}
Let $k$, $T$, and $k'$ be as in Proposition~\ref{char2}, so that $k$ has
characteristic two.  The reduction graph $\Gamma$ of $\XX$ is a tree,
but not a monotonic tree, even though in this case the reduction graph
of every base change is still a tree.  Again we have $\Sha(F,T) \simeq
H^1(k',S)/H^1(k,S) \ne 0$.
\end{enumerate}
\end{example}

In \cite[Theorem~5.10, Corollary~6.5]{HHK3}, it was shown that if $G$ is
a connected linear algebraic group over a semi-global field $F$, and if
$G$ is rational as an $F$-variety, then $\Sha_\PP(F,G)=\Sha_X(F,G)$ is
trivial, whether or not the reduction graph is a tree.
This prompts the following two questions in our situation:

First, given a $k$-torus $T$ and a semi-global field $F$ over $k((t))$,
must $\Sha(F,T)=\Sha_\PP(F,T)$ be trivial if the reduction graph of a normal crossings
model of $F \times_k k'$ is a tree, where $k'$ is a splitting field of
$T$ (i.e., is such that $T_{k'}$ is a split torus and hence rational)?
Example~\ref{non-monotonic Sha}(\ref{ch 2 nontriv Sha}) shows that the
answer to that question is no.  There the field $L$ in
Proposition~\ref{char2} is a splitting field of $T$, but the reduction
graph of the base change to $L$ (and indeed to any finite extension of
$k$) is a tree.

Second, if a torus $T$ is defined over $R$ but is not rational, and if
the reduction graph associated to a normal crossings model of $F$ is not a tree, can $\Sha(F,T)$
still be trivial?  The example in the next section shows that the answer
is yes.

%-----------------------------------------------------------------------------------------------
\subsection{Failure of the exact sequence for nonrational components} \label{triv Sha loop}
%-----------------------------------------------------------------------------------------------

Let $\XX$ be a normal crossings model of a semi-global field $F$ over a complete discretely valued field~$K$ with residue field~$k$ and valuation ring ~$R$. Let $\PP$ be the set of intersection points of the components of the closed fiber of~$\XX$, and let $\Gamma$ be the associated reduction graph.  Let  $T$ be an $R$-torus with flasque resolution $1\rightarrow S\rightarrow Q\rightarrow T\rightarrow 1$.
If the closed fiber consists of a union of copies of ${\mathbb P}^1_k$ meeting at $k$-points, then by Theorem~\ref{Sha via coef sys}(\ref{all P^1}) we have an isomorphism  $\Sha(F,T) \simeq H^1(k,S)^m$, where $m$ is the rank of $H_1(\Gamma_{\operatorname{top}},{\mathbb Z})$; so the displayed sequence in Theorem~\ref{shaP1} is then short exact, not merely right exact.

We now give an example to show that we need not have such an isomorphism if some irreducible component of $X$ is
 not isomorphic to ${\mathbb P}^1_k$. In fact, this example will also show that $\Sha(F,T)$ can be trivial for a non-rational torus even if the reduction graph is not a tree. Therefore, $\Sha(F,T)$ can vanish even if this would not be predicted by that isomorphism. In this example, it will be convenient to use $\rm R$-equivalence; see Section~\ref{reminders}.

\begin{example} \label{loop trivial Sha}
Let $k=k_1$, $T$, $S$ be as in Example~\ref{csv} (with $n=1$).
Then $|H^1(k,S)|=|T(k)/{\rm R}|=2$. In particular, $T$ is not rational. Let $g\in T(k)$ be a representative of the nontrivial class.
Let $C/k$ be a smooth curve with a rational map $f:C\dashrightarrow T$ defined at points $Q_1,Q_2\in C(k)$ such that $f(Q_1)=1\in T(k)$ and $f(Q_2)=g$.
Let $C^*$ be the nodal curve obtained from $C$ by identifying $Q_1$ and
$Q_2$.  By Theorem~3 of \cite{HS99}, there is a regular integral proper
curve $\XX/k[[t]]=R$ with closed fiber $C^*$, since the node can be
deformed locally to the spectrum of a regular complete local ring of
dimension two. Let $P$ be the image of $Q_1,Q_2$ in $C^*$. Blowing up $\XX$ at $P$ gives a normal crossings model whose closed fiber is a union of ${\mathbb P}^1_k$ and $C$, meeting at the two $k$-points $Q_1$ and $Q_2$. Let $U$ be the complement of $\{Q_1,Q_2\}$ in ${\mathbb P}^1_k$, let $V$ be the complement of $\{Q_1,Q_2\}$ in $C$, and let $\Gamma$ be the associated reduction graph. Then by \ref{ShaFromGraph},
 we have $\Sha(F,T)=H^1(\Gamma,\kcoef)$. Recall that the coefficient system $\kcoef$ is given by decorating $Q_i$ with $H^1(k,S)\simeq T(k)/{\rm R}$, decorating $U$ with $H^1(k(U),S)=H^1(k,S)\simeq T(k)/{\rm R}$, and decorating $V$ with $H^1(k(V),S)\simeq T(k(V))/{\rm R}$. Here the isomorphism $H^1(k(U),S)=H^1(k,S)$ holds because $U$ is an affine open in ${\mathbb P}^1_k$, as in the proof of Theorem~\ref{shaP1}.

Now view $f$ above as an element in $T(k(V))=T(k(C))$. By definition, the specialization maps on $V$ at $Q_1$ and $Q_2$ send $f$ to $1$ and to the class of $g$, respectively. It is then easy to see that in the corresponding cochain complex, ${\mathcal C}^0\rightarrow {\mathcal C}^1$ is onto, and hence $\Sha(F,T)=0$. But $\Gamma$  has a loop and $H^1(k,S)\neq 0$, so $\Hom(H_1(\Gamma_{\operatorname{top}},{\mathbb Z}), H^1(k,S))\neq 0$.
\end{example}

Note that the above example also produces an example for the triviality of $\Sha(F,T)$ even though the closed fiber of an associated normal crossings model is not a tree, and $T$ is not rational.

 \providecommand{\bysame}{\leavevmode\hbox to3em{\hrulefill}\thinspace}

%=========================================================
\noindent{\bf Author Information:}\\

\noindent Jean-Louis Colliot-Th\'el\`ene\\
Arithm\'etique et G\'eom\'etrie Alg\'ebrique, Universit\'e Paris-Saclay, CNRS, Laboratoire de\\ Math\'ematiques d'Orsay, 91405, Orsay, France\\
email: jlct@math.u-psud.fr

\medskip

\noindent David Harbater\\
Department of Mathematics, University of Pennsylvania, Philadelphia, PA 19104-6395, USA\\
email: harbater@math.upenn.edu

\medskip

\noindent Julia Hartmann\\
Department of Mathematics, University of Pennsylvania, Philadelphia, PA 19104-6395, USA\\
email: hartmann@math.upenn.edu

\medskip

\noindent Daniel Krashen\\
Department of Mathematics, Rutgers University, Piscataway, NJ 08854-8019, USA\\
email: daniel.krashen@rutgers.edu

\medskip

\noindent R.~Parimala\\
Department of Mathematics and Computer Science, Emory University, Atlanta, GA 30322, USA\\
email: parimala.raman@emory.edu

\medskip

\noindent V.~Suresh\\
Department of Mathematics and Computer Science, Emory University, Atlanta, GA 30322, USA\\
email: suresh.venapally@emory.edu

\medskip

\noindent The authors were supported on an NSF collaborative FRG grant: DMS-1463733 (DH and JH), DMS-1463901 (DK), DMS-1463882 (RP and VS).  Additional support was provided by NSF collaborative FRG grant DMS-1265290 (DH); NSF RTG grant DMS-1344994 and NSF grant DMS-1902144 (DK); NSF DMS-1805439 (DH and JH); NSF DMS-1401319 (RP); NSF DMS-1301785 (VS); and NSF DMS-1801951 (RP and VS).

\end{document}